\documentclass[a4paper]{amsart}
\usepackage{amsmath,amsthm,amsfonts,amssymb,color,graphicx,appendix}
\usepackage{comment}
\usepackage{hyperref}

\usepackage{tikz,pgfplots}
\pgfplotsset{width=8.5cm,compat=1.13}

\numberwithin{equation}{section}
\theoremstyle{plain}

\newtheorem{theorem}{\sc Theorem}[section]

\newtheorem{definition}[theorem]{\sc Definition}
\newtheorem{lemma}[theorem]{\sc Lemma}
\newtheorem{proposition}[theorem]{\sc Proposition}

\theoremstyle{remark}
\newtheorem{remark}[theorem]{\sc Remark}
\newtheorem{example}[theorem]{\sc Example}

\newcommand{\one}{{{\rm 1\mkern-1.5mu}\!{\rm I}}}
\newcommand{\be}{\begin{equation}}
\newcommand{\ee}{\end{equation}}

 \def\Pesssup{\mathop{\mathbb{P}\mbox{-}\,\mathrm{ess\,sup}}}
 
\def\pitwo{\pi^{(x_*,h,1)}}
\def\pieL{\pi^{(x_*,h,\ell)}}
\def\pileft{\overleftarrow\pi}
\def\piright{\overrightarrow\pi}

\def\aeL{\alpha^{(x_*,h,\ell)}}
\def\aleft{\overleftarrow\alpha}
\def\aright{\overrightarrow\alpha}

\def\ep{\epsilon}
\def\epin{\epsilon^{-1}}

\def\corO{}
\def\corOO{}

\begin{document}

\title[Nonconvex homogenization for 1-D controlled RW in random potential]{Nonconvex homogenization for one-dimensional \\controlled random walks in random potential}

\author[A.\ Yilmaz]{Atilla Yilmaz}
\address{Atilla Yilmaz\\ Department of Mathematics\\ Ko\c{c} University\\Rumelifeneri Yolu, Sar\i yer, Istanbul 34450, Turkey}
\email{atillayilmaz@ku.edu.tr}
\urladdr{http://home.ku.edu.tr/$\sim$atillayilmaz}
\thanks{A.\ Yilmaz was partially 
supported by European Union FP7 Marie Curie Career Integration Grant no.\ 322078 and by the BAGEP Award of the Science Academy, Turkey.}

\author[O.\ Zeitouni]{Ofer Zeitouni}
\address{Ofer Zeitouni\\ Faculty of Mathematics\\ Weizmann Institute \\ POB 26, Rehovot 76100\\
Israel\\and
Courant Institute\\ 251 Mercer Street\\
New York, NY 10012\\ USA}
\email{ofer.zeitouni@weizmann.ac.il}
\urladdr{http://wisdom.weizmann.ac.il/$\sim$zeitouni}
\thanks{O.\ Zeitouni was partially supported by an Israel Science Foundation grant} 

\date{May 19, 2017.}

\subjclass[2010]{60K37, 93E20, 35B27.} 
\keywords{Random walk in random potential, stochastic optimal control, 
Hamilton-Jacobi, 
homogenization, corrector, large deviations, tilted free energy.}

\begin{abstract}
We consider a finite horizon stochastic optimal control problem for nearest-neighbor random walk $\{X_i\}$ on the set of integers. The cost function is the expectation of exponential of the path sum of a random stationary and ergodic bounded potential plus $\theta X_n$. The random walk policies are measurable with respect to the random potential, and are adapted, with their drifts  uniformly bounded in magnitude by a parameter $\delta\in[0,1]$. Under natural conditions on the potential, we prove that the normalized logarithm of the optimal cost function converges. The proof is constructive in the sense that we identify asymptotically optimal policies given the value of the parameter $\delta$, as well as the law of the potential. It relies on correctors from large deviation theory as opposed to arguments based on subadditivity which do not seem to work except when $\delta = 0$.

The Bellman equation associated to this control problem is a second-order Hamilton-Jacobi (HJ) stochastic partial difference equation with a separable random Hamiltonian which is nonconvex in 
$\theta$ unless $\delta = 0$. We prove that this equation homogenizes under linear initial data to a first-order HJ deterministic partial differential equation. When $\delta = 0$, the effective Hamiltonian is the tilted free energy of random walk in random potential and it is convex in $\theta$. In contrast, when $\delta = 1$, the effective Hamiltonian is piecewise linear and nonconvex in $\theta$. Finally, when $\delta \in (0,1)$, the effective Hamiltonian is expressed completely in terms of the tilted free energy for the $\delta=0$ case and its convexity/nonconvexity in $\theta$ is characterized by a simple inequality involving $\delta$ and the magnitude of the potential, thereby marking two qualitatively distinct control regimes.
\end{abstract}

\maketitle


\section{Introduction}

\subsection{Controlled random walks in random potential}\label{subsecprob}

Let $(\Omega,\mathcal{F},\mathbb{P})$ be a probability space that is equipped with an ergodic invertible measure-preserving transformation $T:\Omega\to\Omega$. Elements of $\Omega$ are denoted by $\omega$ and referred to as environments. For every $n\in\mathbb{N} = \{1,2,\ldots\}$ and $\delta\in[0,1]$, define
\begin{align*}
\mathcal{P}_n(\delta) = &\left\{\pi = (\pi_0,\pi_1,\ldots,\pi_{n-1}):\pi_i = 
\pi_i(n,\omega,y,\pm 1)\in[0,1]^2 \ 
\corO{\mbox{\rm is $\mathcal{F}$-measurable}},\right.\\
&\left. \ \pi_i(n,\omega,y,-1) + \pi_i(n,\omega,y,1) = 1\ \text{and}
\ |\pi_i(n,\omega,y,1) - \pi_i(n,\omega,y,-1)|\le\delta\ \right.\\
&\left. \ \text{for every $i\in[0,n-1]$, $\omega\in\Omega$ and $y\in\mathbb{Z}$}\right\}.
\end{align*}
Each $\pi\in\mathcal{P}_n(\delta)$ is a (Markov) random walk policy whose drift is uniformly bounded in magnitude by $\delta$. Given any environment $\omega\in\Omega$ and starting point $x\in\mathbb{Z}$, $\pi$ induces a probability measure $P_x^{\pi,\omega}$ on the space of paths $x_{0,n} = (x_0,x_1,\ldots,x_n)\in\mathbb{Z}^{n+1}$ with $x_0 = x$ and $z_{i+1} = x_{i+1} - x_{i}\in \{-1,1\}$, defined by
$$P_x^{\pi,\omega}(X_0 = x_0, X_1 = x_1,\ldots, X_n = x_n) = \prod_{i=0}^{n-1}\pi_i(n,\omega,x_i,z_{i+1}).$$
Expectation under $P_x^{\pi,\omega}$ is denoted by $E_x^{\pi,\omega}$.

Let $V:\Omega\to[0,1]$ be a nonconstant measurable function. $V(T_y\omega)$ is referred to as the potential at the point $y$ in the environment $\omega$. Here and throughout, $T_0 = I$, $T_k = T\circ T_{k-1}$ and $T_{-k} = (T_k)^{-1}$ for $k\in\mathbb{N}$. For every $n\in\mathbb{N}$, $x\in\mathbb{Z}$, $\omega\in\Omega$, $\delta\in[0,1]$, $\beta > 0$ and $\theta\in\mathbb{R}$, let
\begin{equation}\label{beforelimit}
u(n,x,\omega\,|\,\delta,\beta,\theta) = \inf_{\pi\in\mathcal{P}_n(\delta)}\log E_x^{\pi,\omega}\left[e^{\beta\sum_{i=0}^{n-1}V(T_{X_i}\omega) + \theta X_n}\right].
\end{equation}
Note that the left-hand side of \eqref{beforelimit} would not change if we took the infimum on the right-hand side over the larger set of adapted (but not necessarily Markov) random walk policies with drifts still uniformly bounded in magnitude by $\delta$. 
(See \cite[Proposition 11.7]{BerShr1978}.) 

%
%

\subsection{Overview of our results}\label{kereviziyi}

We show in Section \ref{bankerbilo} that, under natural assumptions, for $\mathbb{P}$-a.e.\ $\omega$ the limit
\begin{equation}\label{homlimit}
u_o(t,x\,|\,\delta,\beta,\theta) = \lim_{\ep\to0}\ep u\left([\ep^{-1}t],[\ep^{-1}x],\omega\,|\,\delta,\beta,\theta\right)
\end{equation}
exists for every $t > 0$ and $x\in\mathbb{R}$ (where $[\cdot]$ denotes the floor function), and it is of the form
\begin{equation}\label{simdigel}
u_o(t,x\,|\,\delta,\beta,\theta) = t\overline H_{\delta,\beta}(\theta) + \theta x.
\end{equation}
$\overline H_{\delta,\beta}(\theta) := u_o(1,0\,|\,\delta,\beta,\theta)$ is a deterministic quantity for which
we provide a formula. In fact, for $\delta>0$ we express $\overline H_{\delta,\beta}(\cdot)$ completely in terms of $\overline H_{0,\beta}(\cdot)$. The existence of the latter was already known (see Section \ref{subseczeroc}) and can be shown via subadditivity (see Appendix \ref{app_subadd}). However, there is no subadditivity to be exploited when $\delta>0$, so instead we develop a constructive approach. In particular, in Section \ref{asop} we identify asymptotically optimal policies (as $n\to\infty$) for the control problem in \eqref{beforelimit}.

We make two observations. First, the Bellman equation associated to the control problem in \eqref{beforelimit} is a second-order Hamilton-Jacobi (HJ) stochastic partial difference equation (see \eqref{oriHJ2}). Second, the function $u_o(t,x) = u_o(t,x\,|\,\delta,\beta,\theta)$ (given in \eqref{simdigel}) satisfies the following first-order HJ deterministic partial differential equation:
$$\frac{\partial u_o}{\partial t}(t,x) = \overline H_{\delta,\beta}\left(\frac{\partial u_o}{\partial x}(t,x)\right).$$
Due to the limit in \eqref{homlimit} under an appropriate scaling of time and space, the former equation (with linear initial data) is said to homogenize to the latter one. See Section \ref{homressec} for details and also Section \ref{syurdubey} for related results from the homogenization literature. Therefore, throughout the paper, $\overline H_{\delta,\beta}(\theta)$ will be referred to as the effective Hamiltonian.


%
%

\subsection{Assumptions on the potential}\label{subsecass}

Since the potential inside the expectation on the right-hand side of \eqref{beforelimit} is scaled by $\beta$, there is no loss of generality in assuming that
\begin{equation}\label{ass_wlog}
\text{the essential infimum (resp.\ supremum) of $V(\omega)$ under $\mathbb{P}$ is $0$ (resp.\ $1$).}
\end{equation}
Our results will further require the existence of arbitrarily long finite intervals where the potential is uniformly close to its essential infimum (resp.\ supremum). In order to make this condition precise, we introduce two terms.
\begin{definition}\label{pandef}
	For any $\omega\in\Omega$ and $h\in(0,1)$, an interval $[k,\ell]\subset\mathbb{Z}$ is said to be an $h$-valley (resp.\ $h$-hill) if $V(T_y\omega) \le h$ (resp.\ $V(T_y\omega) \ge h$) for every $y\in[k,\ell]$.
\end{definition}
With this terminology, we will assume that
\begin{equation}\label{ass_pan}
\mathbb{P}(\text{$[0,\ell]$ is an $h$-valley}) > 0\quad\text{and}\quad\mathbb{P}(\text{$[0,\ell]$ is an $h$-hill}) > 0\quad\text{for every $h\in(0,1)$ and $\ell\in\mathbb{N}$.}
\end{equation}
Note that this assumption does not imply \corO{that the environment 
is mixing},
\corO{as Example \ref{ex-nonmix} below shows.}

\begin{example}
  \label{ex-first}
	Let $\Omega = [0,1]^\mathbb{Z}$ and $\mathcal{F}$ the Borel $\sigma$-algebra on $\Omega$. Define $T:\Omega\to\Omega$ by $(T\omega)_y = \omega_{y+1}$ for any $\omega = (\omega_y)_{y\in\mathbb{Z}}\in\Omega$. Assume that
	\begin{itemize}
		\item [(i)] $\mathbb{P}$ is a probability measure on $(\Omega,\mathcal{F})$ that is stationary and ergodic under $T$, and
		\item [(ii)] there exists a Borel probability measure $\mu$ on $[0,1]$ such that the product measure $\prod_{y\in\mathbb{Z}}\mu$ is absolutely continuous with respect to $\mathbb{P}$ on $\mathcal{F}_{0,\ell} = \sigma\{\omega_0,\ldots,\omega_\ell\}$ for every $\ell\in\mathbb{N}$.
	\end{itemize}
	Consider the function $V:\Omega\to[0,1]$ given by $V(\omega) = \omega_0$. Then, \eqref{ass_wlog} is equivalent to $\mu$ having full support, in which case \eqref{ass_pan} holds by the assumption of absolute continuity.
\end{example}
\begin{example}
  \label{ex-nonmix}
  \corO{
With $\Omega$, $\mathcal{F}$, $T$ and $V$ as 
in Example \ref{ex-first},
let $\{s\}\cup\{\alpha_k:\,k\in\mathbb{Z}\}$ be an i.i.d.\ collection of ($\{0,1\}$-valued) Bernoulli trials with success probability $1/2$. Define $\omega = (\omega_y)_{y\in\mathbb{Z}}$ by setting
$$\omega_{s+2k-1} = \omega_{s+2k} = \alpha_k$$
for every $k\in\mathbb{Z}$. This induces a probability measure $\mathbb{P}$ on $(\Omega,\mathcal{F})$. It is clear that $\mathbb{P}$ is stationary and ergodic under $T$. Moreover, \eqref{ass_wlog} and \eqref{ass_pan} trivially 
hold. However, $\omega$ is not even weakly mixing under $\mathbb{P}$,
since an elementary computation shows that with $A=\{\omega_{-1}=\omega_0\}$
one has that $\mathbb{P}(A\cap T^{-2k} A)=5/8$ for all $k\neq 0$ while 
$(\mathbb{P}(A))^2=9/16$.}
\end{example}

\subsection{Special case: No control}\label{subseczeroc}

If $\delta = 0$, then $\mathcal{P}_n(\delta)$ is a singleton whose unique element satisfies $\pi_i(n,\omega,y,\pm1) \equiv 1/2$ and induces simple symmetric random walk (SSRW) on $\mathbb{Z}$. In this case, we simplify the notation and write $P_x$ (resp.\ $E_x$) instead of $P_x^{\pi,\omega}$ (resp.\ $E_x^{\pi,\omega}$).

\begin{theorem}[No control]\label{thmno}
Assume \eqref{ass_wlog} and \eqref{ass_pan}. If $\delta = 0$, $\beta > 0$ and $\theta\in\mathbb{R}$, then for $\mathbb{P}$-a.e.\ $\omega$ the limit in \eqref{homlimit} exists for every $t>0$ and $x\in\mathbb{R}$. Moreover, \eqref{simdigel} holds and the effective Hamiltonian is given by
\begin{equation}\label{nolimit}
\overline H_{0,\beta}(\theta) = \Lambda_\beta(\theta) := \lim_{n\to\infty}\frac1{n}\log E_0\left[e^{\beta\sum_{i=0}^{n-1}V(T_{X_i}\omega) + \theta X_n}\right],
\end{equation}
the so-called tilted free energy.
\end{theorem}

The existence of the tilted free energy was shown in several previous works in much greater generality. 
Zerner \cite{Zer1998} considered nearest-neighbor random walks (RWs) in i.i.d.\ random potential on $\mathbb{Z}^d$ (with any $d\ge1$) and gave a subadditivity argument that proves the existence of certain Lyapunov exponents which in turn imply a large deviation principle (LDP) for the position of the walk. Then, Flury \cite{Flu2007} used Zerner's large deviation result to show the existence of the tilted free energy in the same setting. These two papers built upon earlier work by Sznitman \cite{Szn1994} on Brownian motion in a Poissonian potential on $\mathbb{R}^d$. By another subadditivity argument, Varadhan \cite{Var2003} bypassed Lyapunov exponents and directly established a similar LDP for a closely related model, namely nearest-neighbor RW in stationary and ergodic (not necessarily i.i.d.) random environment on $\mathbb{Z}^d$. It is easy to adapt Varadhan's argument to give a short proof of the existence of the tilted free energy for RW in random potential on $\mathbb{Z}^d$. We do this \corO{in a more general setup
  in Theorem \ref{pargoz} of} 
Appendix \ref{app_subadd} for the sake of completeness 
and with future use in mind. 
There are alternative proofs of Theorem \ref{pargoz} which provide variational formulas for the tilted free energy \cite{Yil2009,RasSepYil2013,RasSep2014,RasSepYil2017}. See Remark \ref{furrefvar} for details.

In Section \ref{zerocontrol}, we will take advantage of our one-dimensional setting to present a self-contained proof of Theorem \ref{thmno} (which is not based on subadditivity) and give an implicit (non-variational) formula 
for the tilted free energy $\Lambda_\beta(\theta)$. 
We will also show some properties of $\Lambda_\beta(\theta)$ as a function of $\beta$ and $\theta$ (see Proposition \ref{temelsaf}). In particular, if $\delta = 0$, then the effective Hamiltonian $\overline H_{0,\beta}(\theta) = \Lambda_\beta(\theta)$ is convex in $\theta$ for every $\beta>0$.


%
%

\section{Results}\label{buyukressec}

\subsection{The effective Hamiltonian}\label{bankerbilo}

As we present below, for $\mathbb{P}$-a.e.\ $\omega$ the limit in \eqref{homlimit} exists for every $t>0$ and $x\in\mathbb{R}$ under the assumptions \eqref{ass_wlog} and \eqref{ass_pan}. Recall from Section \ref{subseczeroc} that the special case of no control (i.e., $\delta = 0$) is studied in detail in Section \ref{zerocontrol}. 
The other extreme case is $\delta =1$, i.e., when we can fully control the trajectory of the particle performing the walk. The analysis of the latter case involves the same approach as the intermediate case $\delta\in(0,1)$ but it is technically simpler, so we present it first.

\begin{theorem}[Full control]\label{thmfull}
Assume \eqref{ass_wlog} and \eqref{ass_pan}. If $\delta = 1$, $\beta > 0$ and $\theta\in\mathbb{R}$, then for $\mathbb{P}$-a.e.\ $\omega$ the limit in \eqref{homlimit} exists for every $t>0$ and $x\in\mathbb{R}$. Moreover, \eqref{simdigel} holds and the effective Hamiltonian is given by
\begin{equation}\label{fulllimit}
\overline H_{1,\beta}(\theta) = \begin{cases}
0&\ \text{if}\ |\theta| < \beta\mathbb{E}[V(\cdot)],\\
\beta\mathbb{E}[V(\cdot)] - |\theta|&\ \text{if}\ |\theta| \ge \beta\mathbb{E}[V(\cdot)].
\end{cases}
\end{equation}
\end{theorem}

When $\delta\in(0,1)$, we can only partially control the trajectory of the particle.
In order to give a tidy formula for $\overline H_{\delta,\beta}(\theta)$, we introduce the parameter
\begin{equation}\label{parasi}
c = \frac1{2}\log\left(\frac{1 + \delta}{1 - \delta}\right). 
\end{equation}
The comparison of $\beta$ and $\log\cosh(c)$ 
(or equivalently of $\sqrt{1 - e^{-2\beta}}$ and $\delta$) turns out to play a critical role, giving rise to two qualitatively distinct regimes to which we will refer below as weak control and strong control.

\begin{theorem}[Weak control]\label{thmweak}
Assume \eqref{ass_wlog} and \eqref{ass_pan}. If $\delta\in(0,1)$, $\beta \ge \log\cosh(c)$ and $\theta\in\mathbb{R}$, then for $\mathbb{P}$-a.e.\ $\omega$ the limit in \eqref{homlimit} exists for every $t>0$ and $x\in\mathbb{R}$. Moreover, \eqref{simdigel} holds and the effective Hamiltonian is given by
\begin{equation}\label{weaklimit}
\overline H_{\delta,\beta}(\theta) = \begin{cases}
\beta - \log\cosh(c)&\ \text{if}\ |\theta| < c,\\
\Lambda_\beta(|\theta| - c) - \log\cosh(c)&\ \text{if}\ |\theta| \ge c.
\end{cases}
\end{equation}
\end{theorem}


\begin{theorem}[Strong control]\label{thmstrong}
Assume \eqref{ass_wlog} and \eqref{ass_pan}. If $\delta\in(0,1)$, $\beta < \log\cosh(c)$ and $\theta\in\mathbb{R}$, then for $\mathbb{P}$-a.e.\ $\omega$ the limit in \eqref{homlimit} exists for every $t>0$ and $x\in\mathbb{R}$. Moreover, \eqref{simdigel} holds, there exists a unique $\bar\theta(\beta,c)\in(0,c)$ such that
$$\Lambda_\beta(\bar\theta(\beta,c)-c) = \log\cosh(c),$$
and the effective Hamiltonian is given by
\begin{equation}\label{stronglimit}
\overline H_{\delta,\beta}(\theta) = \begin{cases}
0&\ \text{if}\ |\theta| < |\bar\theta(\beta,c)|,\\
\Lambda_{\beta}(|\theta| - c) - \log\cosh(c)&\ \text{if}\ |\theta| \ge |\bar\theta(\beta,c)|.
\end{cases}
\end{equation}
\end{theorem}

Substituting $c=0$ in \eqref{weaklimit} reproduces the formula in \eqref{nolimit}. Similarly, taking $c\to\infty$ in \eqref{stronglimit} reproduces the formula in \eqref{fulllimit} by Proposition \ref{temelsaf}(d).

\subsection{Asymptotically optimal policies}\label{asop}

The proofs of Theorems \ref{thmfull}, \ref{thmweak} and \ref{thmstrong} are constructive in the sense that we identify RW policies that are asymptotically optimal in each case. We introduce these policies below.

For every $h\in(0,1)$, $\ell\in\mathbb{N}$ and $\mathbb{P}$-a.e.\ $\omega$, we choose an $h$-valley (recall from Definition \ref{pandef}) of the form $[x_*-\ell,x_*+\ell-1]$ with some $x_*\in\mathbb{Z}$ that is suitably close to the starting point of the RW (see Remark \ref{dundundur} for details). 
We define a RW policy $\pieL$ by setting
\begin{equation}\label{pengol}
\pieL_i(n,\omega,y,1) = \begin{cases}\frac{1 + \delta}{2}&\ \text{if $y < x_*$},\\
\frac{1 - \delta}{2}&\ \text{if $y \ge x_*$}.\end{cases}
\end{equation}
Note that it is a bang-bang policy (see, e.g., \cite{Art1980}). We also consider the spatiotemporally constant bang-bang policies $\pileft$ and $\piright$ given by
\begin{equation}\label{sagsol}
\pileft_i(n,\omega,y,1) \equiv \frac{1 - \delta}{2}\quad\text{and}\quad\piright_i(n,\omega,y,1) \equiv \frac{1 + \delta}{2}.
\end{equation}

In each of the three regimes of weak, strong and full control, the graph of $\overline H_{\delta,\beta}(\theta)$ against $\theta$ has a flat region centered at the origin (see Figure \ref{matrixfigure}). When $\theta$ is in this flat region, it will turn out that the infimum in \eqref{beforelimit} can be taken over the set of $\pieL$ with arbitrarily small $h\in(0,1)$ and arbitrarily large $\ell\in\mathbb{N}$. Doing so creates a $o(n)$ difference which does not change the limit in \eqref{homlimit}. (When $\delta = 1$, it suffices to take $\ell = 1$.) On the other hand, when $\theta$ is to the right (resp.\ left) of the flat region centered at the origin, it will turn out that the infimum in \eqref{beforelimit} is asymptotically attained at $\pileft$ (resp.\ $\piright$) up to a $o(n)$ term as $n\to\infty$.

Even though the regimes of weak and strong control share a common class of asymptotically optimal policies at (say) $\theta = 0$, namely the policies $\pieL$, the value of $\overline H_{\delta,\beta}(0)$ is different in these two cases (see Theorems \ref{thmweak} and \ref{thmstrong}), which is caused by the difference in the large deviation behavior of the walk under $\pieL$. In this sense, our optimal control problem can be thought of as a two-person game where the players are (i) the controller and (ii) the particle exhibiting atypical behavior. This point will become clear in the proofs.

\subsection{Homogenization of the Bellman equation}\label{homressec}

For every $n\in\mathbb{N}$, $x\in\mathbb{Z}$, $\omega\in\Omega$, $\delta\in[0,1]$, $\beta > 0$ and $\theta\in\mathbb{R}$, we write $u(n,x,\omega) = u(n,x,\omega\,|\,\delta,\beta,\theta)$ for notational brevity and then arrange \eqref{beforelimit} as
$$e^{u(n,x,\omega)} = \inf_{\pi\in\mathcal{P}_n(\delta)} E_x^{\pi,\omega}\left[e^{\beta\sum_{i=0}^{n-1}V(T_{X_i}\omega) + \theta X_n}\right].$$
Decomposing the expectation in the corresponding expression for $e^{u(n+1,x,\omega)}$ with respect to the first step of the controlled walk and applying the Bellman principle gives
\begin{equation}\label{kisamioldu}
u(n+1,x,\omega) = \beta V(T_x\omega) + \inf_{q\in[\frac{1-\delta}{2},\frac{1+\delta}{2}]}\log\left(q e^{u(n,x+1,\omega)} + (1-q)e^{u(n,x-1,\omega)}\right).
\end{equation}
Due to linearity in the parameter $q$ and the monotonicity of the logarithm function, the infimum on the right-hand side of \eqref{kisamioldu} is attained at $\frac{1 - \delta}{2}$ or $\frac{1 + \delta}{2}$. (Therefore, the infimum in \eqref{beforelimit} can be taken over the set of bang-bang policies. We will recapitulate and use this in Section \ref{hssira}.) Evaluating this infimum, switching to the parameter $c$ introduced in \eqref{parasi} in the case $\delta\in(0,1)$, and finally substracting $u(n,x,\omega)$ from both sides of \eqref{kisamioldu}, we deduce that
\begin{equation}\label{oriHJ2}
\nabla_1 u(n,x,\omega) = \begin{cases}
\frac1{2}\Delta_2 u(n,x,\omega) + \log\cosh(\nabla_2 u(n,x,\omega)) + \beta V(T_x\omega)&\ \text{if $\delta = 0$,}\\
\frac1{2}\Delta_2 u(n,x,\omega) + \log\cosh(|\nabla_2 u(n,x,\omega)| - c) - \log\cosh(c) + \beta V(T_x\omega)&\ \text{if $\delta\in(0,1)$,}\\
\frac1{2}\Delta_2 u(n,x,\omega) - |\nabla_2 u(n,x,\omega)| + \beta V(T_x\omega)&\ \text{if $\delta = 1$.}\end{cases}
\end{equation}
Here, we use the notation
\begin{align*}
\nabla_1 u(n,x,\omega) &= u(n+1,x,\omega) - u(n,x,\omega),\\
\nabla_2 u(n,x,\omega) &= \frac1{2}[u(n,x+1,\omega) - u(n,x-1,\omega)]\quad\text{and}\\
\Delta_2 u(n,x,\omega) &= u(n,x-1,\omega) + u(n,x+1,\omega) - 2u(n,x,\omega)
\end{align*}
for these difference operators. Hence, \eqref{beforelimit} solves a second-order HJ stochastic partial difference equation, subject to the linear initial condition $u(0,x,\omega) = \theta x$, with the following separable random Hamiltonian:
\begin{equation}\label{orhemi}
H_{\delta,\beta}(\theta,x,\omega) = K_\delta(\theta) + \beta V(T_x\omega),\quad K_\delta(\theta) = \begin{cases}
\log\cosh(\theta)&\ \text{if $\delta = 0$,}\\
\log\cosh(|\theta| - c) - \log\cosh(c)&\ \text{if $\delta\in(0,1)$,}\\
-|\theta|&\ \text{if $\delta = 1$.}\end{cases}
\end{equation}

For every $\ep>0$, $t\ge0$, $x\in\mathbb{R}$ and $\omega\in\Omega$, let
$$u_\ep(t,x,\omega) = \ep u\left([\ep^{-1}t],[\ep^{-1}x],\omega\right).$$
After appropriate substitutions, \eqref{oriHJ2} becomes
\begin{equation}\label{resHJ2}
\nabla_1^\ep u_\ep(t,x,\omega) = \frac{\ep}{2}\Delta_2^\ep u_\ep(t,x,\omega) + H_{\delta,\beta}(\nabla_2^\ep u_\ep(t,x,\omega),[\ep^{-1}x],\omega),
\end{equation}
where
\begin{align*}
\nabla_1^\ep u_\ep(t,x,\omega) &= \ep^{-1}[u_\ep(t+\ep,x,\omega) - u_\ep(t,x,\omega)],\\
\nabla_2^\ep u_\ep(t,x,\omega) &= (2\ep)^{-1}[u_\ep(t,x+\ep,\omega) - u_\ep(t,x-\ep,\omega)]\quad\text{and}\\
\Delta_2^\ep u_\ep(t,x,\omega) &= \ep^{-2}[u_\ep(t,x-\ep,\omega) + u_\ep(t,x+\ep,\omega) - 2u_\ep(t,x,\omega)].
\end{align*}

As we mentioned in Section \ref{kereviziyi}, the function
$u_o(t,x) = u_o(t,x\,|\,\delta,\beta,\theta) = t\overline H_{\delta,\beta}(\theta) + \theta x$ solves 
\begin{equation}\label{ihlaravad}
\frac{\partial u_o}{\partial t}(t,x) = \overline H_{\delta,\beta}\left(\frac{\partial u_o}{\partial x}(t,x)\right), \quad
\corO{u_o(0,x) = \theta x.}
\end{equation}
Our final result combines Theorems \ref{thmno}, \ref{thmfull}, \ref{thmweak} and \ref{thmstrong}, improves the pointwise convergence (in $t>0$ and $x\in\mathbb{R}$) in their statements to uniform convergence on compact sets.

\begin{theorem}[Homogenization with linear initial data]\label{unifhom2}
Assume \eqref{ass_wlog} and \eqref{ass_pan}. If $\delta\in[0,1]$, $\beta>0$ and $\theta\in\mathbb{R}$, then for $\mathbb{P}$-a.e.\ $\omega$ the function $u_\ep(\cdot,\cdot,\omega)$ converges to $u_o(\cdot,\cdot)$ as $\ep\to 0$, uniformly on compact subsets of $[0,\infty)\times\mathbb{R}$, 
  \corO{with the effective Hamiltonian $\overline H_{\delta,\beta}(\theta)$ given in \eqref{nolimit}, \eqref{fulllimit}, \eqref{weaklimit} and \eqref{stronglimit} in the cases of no, full, weak and strong control, respectively.}
\end{theorem}
\corO{In the language of homogenization theory, 
  Theorem \ref{unifhom2} says that
 the second-order HJ stochastic partial difference equation in 
 \eqref{resHJ2} with the initial condition 
 $u_\ep(0,x,\omega) = \theta\ep[\ep^{-1}x]$ homogenizes to the 
 first-order HJ deterministic partial differential 
 equation in \eqref{ihlaravad}. }

The original Hamiltonian $H_{\delta,\beta}(\theta,x,\omega) = K_\delta(\theta) + \beta V(T_x\omega)$ (given in \eqref{orhemi}) is convex in $\theta$ in the case of no control, and it is nonconvex in the cases of weak, strong and full control. On the other hand, the effective Hamiltonian $\overline H_{\delta,\beta}(\theta)$ is convex in $\theta$ in the cases of no and weak control, and it is nonconvex in the cases of strong and full control. (See Figure \ref{matrixfigure}.) We summarize this as follows:
\begin{equation}\label{marcak}
\text{$\overline H_{\delta,\beta}(\theta)$ is convex in $\theta$}\ \iff\ \log\cosh(c)\le\beta \iff\ \delta\le\sqrt{1-e^{-2\beta}}.
\end{equation}

\begin{remark}\label{remzik}
Observe that
\begin{itemize}
	\item [(i)] $\log\cosh(c)$ is equal to the depth of the wells in the graph of $K_\delta(\theta)$ against $\theta$, and
	\item [(ii)] 
	$\beta = \sup\{\beta V(T_x\omega) - \beta V(T_y\omega):\, x,y\in\mathbb{Z}\}$
	for $\mathbb{P}$-a.e.\ $\omega$ (by \eqref{ass_wlog}).
\end{itemize}
Therefore, the first equivalence in \eqref{marcak} is a purely geometric characterization of the convexity of the effective Hamiltonian in terms of the original Hamiltonian.
\end{remark}

\begin{figure}
\includegraphics{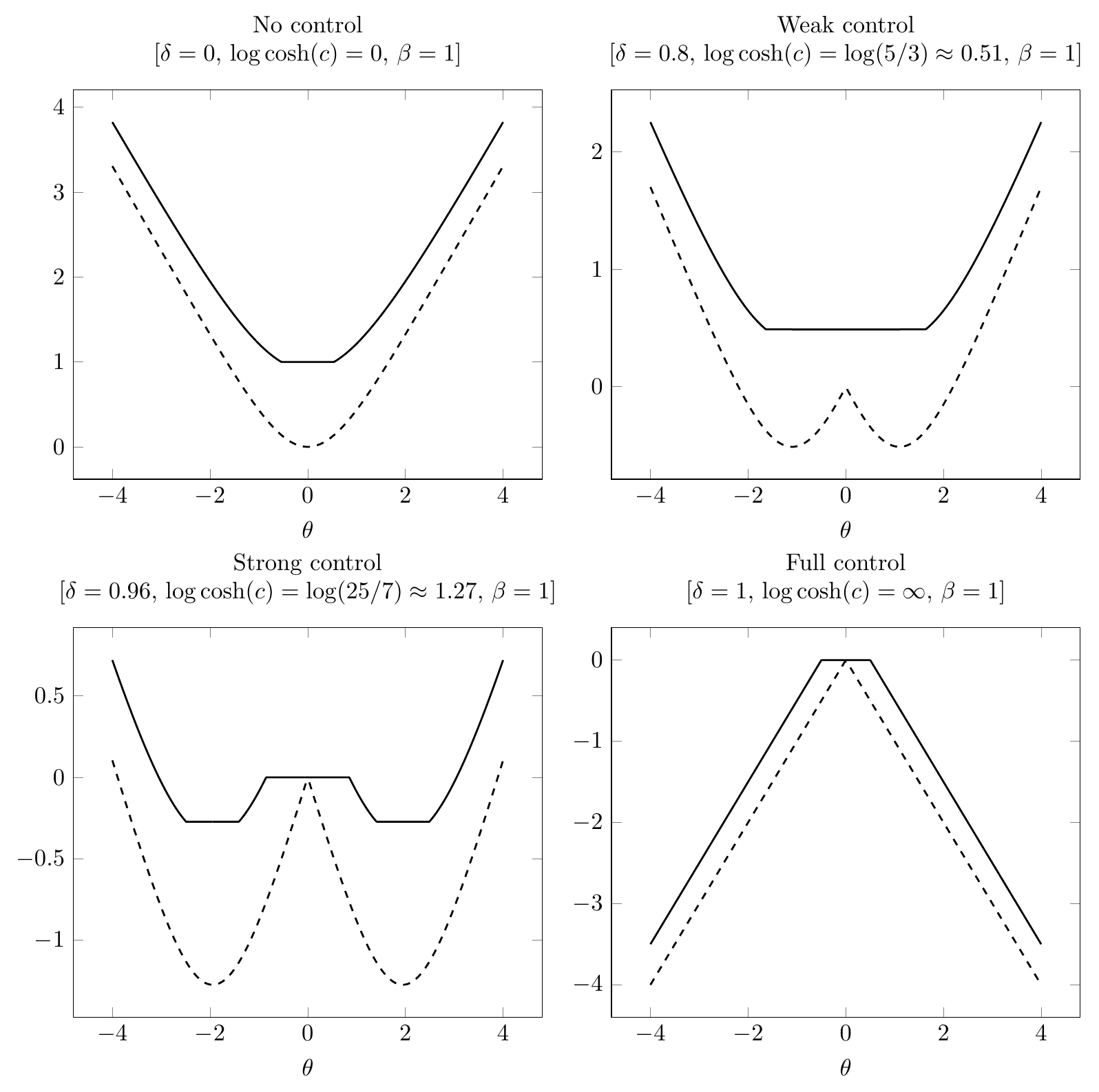}
\caption{Representative graphs of $K_\delta(\theta)$ (dashed) and $\overline H_{\delta,\beta}(\theta)$ (solid) against $\theta$ in each control regime when $\beta = 1$. There is weak control if and only if $0<\log\cosh(c)\le\beta = 1$ if and only if $0<\delta\le\sqrt{1 - e^{-2\beta}}\approx 0.93$. (To sketch these graphs, we assumed without loss of generality that $\mathbb{E}[V(\cdot)] = 0.5$.)}
\label{matrixfigure}
\end{figure}

\subsection{Some previous results on the homogenization of HJ equations}\label{syurdubey}

Recall from Section \ref{subseczeroc} that the existence of the tilted free energy was already shown for a general class of RWs in random potentials on $\mathbb{Z}^d$ with any $d\ge1$ (see \cite{Zer1998,Flu2007,Yil2009,RasSepYil2013} and also Remark \ref{furrefvar}). In light of Theorem \ref{thmno}, this existence result can be seen as ``pointwise homogenization" at $(t,x)=(1,0)$ for a second-order HJ stochastic partial difference equation with linear initial data, where the Hamiltonian is given by the tilted free energy and hence convex in $\theta$. It is not hard to improve the pointwise convergence at $(t,x)=(1,0)$ to uniform convergence on compact subsets of $[0,\infty)\times\mathbb{R}$ (as we do so in Theorems \ref{thmno} and \ref{unifhom2} in the one-dimensional case with no control). To the best of our knowledge, there are no other previous results on the homogenization of second-order HJ stochastic partial difference equations.

There is a rich literature on the continuous analog of our discrete setting with no control and its suitable generalizations. Sznitman's work \cite{Szn1994} on large deviations for Brownian motion in a Poissonian potential on $\mathbb{R}^d$ employs the subadditive ergodic theorem and gives the first example of ``pointwise homogenization" of a second-order HJ stochastic partial differential equation (PDE) with linear initial data, where the Hamiltonian is quadratic (and hence convex) in $\theta$. Homogenization of second-order HJ stochastic PDEs (with general uniformly continuous initial data which is what is meant by default) was later established in \cite{LioSou2005} (using the subadditive ergodic theorem) and independently in \cite{KosRezVar2006} (using the ergodic and minimax theorems) for wide classes of Hamiltonians that are convex in $\theta$. In fact, as we mention in Section \ref{subseczeroc} and Remark \ref{furrefvar}, the existence of the tilted free energy for RWs in random potentials on $\mathbb{Z}^d$ was shown in \cite{Zer1998,Flu2007} and then \cite{Yil2009,RasSepYil2013} by building upon the ideas in \cite{Szn1994} and \cite{KosRezVar2006}, respectively. For further details and references on the homogenization of (first- and second-order) HJ stochastic PDEs with convex Hamiltonians, see \cite{Kos2007}.

There are also several works that prove homogenization for certain HJ stochastic PDEs with nonconvex Hamiltonians in arbitrary dimensions. In the second-order case (which is relevant to our setting), the work of Fehrman \cite{Feh_preprint} covers a class of ``level-set convex" Hamiltonians, whereas Armstrong and Cardaliaguet \cite{ArmCar_preprint} consider Hamiltonians that satisfy a finite range of dependence condition and are homogeneous in $\theta$. The Hamiltonian $H_{\delta,\beta}(\theta,x,\omega)$ in our setting (which is given in \eqref{orhemi} and is nonconvex in $\theta$ when $\delta\in(0,1]$) satisfies none of these conditions.

In the first-order case, Armstrong, Tran and Yu \cite{ArmTraYu2015} prove homogenization for a HJ stochastic PDE in arbitrary dimensions, where the Hamiltonian is of the form $H(\theta) + V(T_x\omega)$ with the specific choice $H(\theta) = (|\theta|^2 - 1)^2$. In a subsequent work \cite{ArmTraYu2016}, the same authors extend this result to any coercive $H(\theta)$ in one dimension. They also notice the relationship between (i) the convexity of the effective Hamiltonian $\overline H(\theta)$ and (ii) the size of the oscillations of $V(T_x\omega)$ in comparison to the depth of the wells of $H(\theta)$ (which is similar to Remark \ref{remzik}). Moreover, they give an implicit formula for $\overline H(\theta)$ under additional assumptions (see \cite[Lemma 5.2]{ArmTraYu2016}). The proofs in \cite{ArmTraYu2016} rely on the existence of sublinear correctors in one dimension which is parallel to our approach (see Section \ref{subsecstrategy} for a summary of our proofs), but are otherwise quite different since they use (first-order) nonlinear PDE techniques. The main homogenization result in \cite{ArmTraYu2016} is extended by Gao \cite{Gao_preprint} to general (i.e., not necessarily separable) coercive Hamiltonians in one dimension.

In a recent paper, Davini and Kosygina \cite{DavKos_preprint} consider first- and second-order HJ stochastic PDEs in arbitrary dimensions. Using a variant of the perturbed test function method which is originally due to Evans \cite{Eva1989}, they prove that ``pointwise homogenization" at $(t,x) = (1,0)$ with linear initial data in fact implies homogenization with general uniformly continuous initial data. We expect that this result can be adapted to our discrete setting, too, and in particular extend Theorem \ref{unifhom2} to uniformly continuous initial data. However, we did not pursue this direction since our starting point is controlled RWs in random potential for which the corresponding initial data is linear.

As an application of their main result in \cite{DavKos_preprint}, Davini and Kosygina show homogenization for nonconvex Hamiltonians of the following form in one dimension: there exist finitely many $\theta_1,\ldots,\theta_n$ such that the Hamiltonian is constant at these values and it is convex in $\theta$ on each of the intervals $(-
\infty,\theta_1)$, $(\theta_1,\theta_2)$,\ \ldots, $(\theta_{n-1},\theta_n)$, $(\theta_n,+\infty)$. Due to the random additive term $\beta V(T_x\omega)$ in \eqref{orhemi}, the Hamiltonian $H_{\delta,\beta}(\theta,x,\omega)$ in our setting does not have this form.

Finally, Ziliotto \cite{Zil_preprint} proves,
by giving a counterexample,
that first-order HJ stochastic PDEs do not always homogenize. His counterexample comes from a zero-sum differential game in two dimensions. The Hamiltonian is coercive, Lipschitz continuous and (of course) nonconvex in $\theta$. The environment is stationary and ergodic (in fact, slowly mixing). Even though there are currently no such counterexamples in the second-order case, Ziliotto's work suggests that one cannot prove homogenization results by purely qualitative arguments based on subadditivity when the Hamiltonian is nonconvex, and some kind of constructive approach (such as ours in this paper) is needed.

\section{Summary of the proofs}\label{subsecstrategy}

In order to convey the essence and strategy of the proofs of Theorems \ref{thmno}, \ref{thmfull}, \ref{thmweak} and \ref{thmstrong} to the reader at a relatively early stage in the paper, we provide here an overview without giving full details, proper justifications or references (which can all be found in the subsequent sections).

\subsection{No control}\label{summ1}

Similar to $u_o(t,x\,|\,\delta,\beta,\theta)$ in \eqref{homlimit} with $\delta=0$, we define $\Lambda_\beta^L(\theta,t,x)$ and $\Lambda_\beta^U(\theta,t,x)$ in \eqref{kosk1} and \eqref{kosk2} but via $\liminf$ and $\limsup$, respectively.

For every $h\in(0,1)$ and $\ell\in\mathbb{N}$, there is an $h$-hill of the form $[x^*-\ell,x^*+\ell-1]$ that is suitably close to the starting point of the RW. The distance is controlled by a small parameter $a>0$. We consider the event that the particle marches deterministically to $x^*$ and then spends the rest of the time in this $h$-hill, which gives the lower bound \begin{equation}\label{oakar}
\Lambda_\beta^L(\theta,t,x) \ge t\beta + \theta x
\end{equation}
after taking $a\to0$, $h\to1$ and $\ell\to\infty$.
 
If $\Lambda_\beta^U(\theta,t,x) = t\beta + \theta x$, then
$$u_o(t,x\,|\,0,\beta,\theta) = \Lambda_\beta^L(\theta,t,x) = \Lambda_\beta^U(\theta,t,x) = t\beta + \theta x = t\Lambda_\beta(\theta) + \theta x$$
and we are done. Otherwise, we construct a bounded and centered cocycle $F_{\beta,\theta}:\Omega\times\{-1,1\}\to\mathbb{R}$ (referred to as the corrector) that satisfies
\begin{equation}\label{aydenti}
e^\lambda = \frac1{2}e^{\beta V(\omega) + \theta + F_{\beta,\theta}(\omega,1)} + \frac1{2}e^{\beta V(\omega) - \theta + F_{\beta,\theta}(\omega,-1)}
\end{equation}
for some $\lambda>\beta$. The sums $\sum F_{\beta,\theta}(T_{x_i}\omega,z_{i+1})$ over nearest-neighbor paths are uniformly sublinear in the number of steps. We use these sublinear path sums to modify the exponential expectations on the right-hand sides of \eqref{kosk1} and \eqref{kosk2} without changing the values of $\Lambda_\beta^L(\theta,t,x)$ and $\Lambda_\beta^U(\theta,t,x)$. After this modification, it follows from  a repeated application of \eqref{aydenti} that
$$u_o(t,x\,|\,0,\beta,\theta) = \Lambda_\beta^L(\theta,t,x) = \Lambda_\beta^U(\theta,t,x) = t\lambda + \theta x = t\Lambda_\beta(\theta) + \theta x.$$
This completes the proof of Theorem \ref{thmno}. We also deduce that $\Lambda_\beta(\theta)\ge\beta$.

\subsection{Full control}\label{summ2}

Similar to $u_o(t,x\,|\,\delta,\beta,\theta)$ in \eqref{homlimit} with $\delta\in(0,1]$, we define $\overline H_{\delta,\beta}^U(\theta,t,x)$ and $\overline H_{\delta,\beta}^L(\theta,t,x)$ in \eqref{homlimsup} and \eqref{homliminf} but via $\limsup$ and $\liminf$, respectively.

For every $h\in(0,1)$, there is an $h$-valley of the form $[x_*-1,x_*]$ that is suitably close to the starting point of the RW. The distance is controlled by $a>0$ as in Section \ref{summ1}. Under the policy $\pitwo$ (given in \eqref{pengol}), the particle marches deterministically to $x_*$ and is then confined to $[x_*-1,x_*]$ for the rest of the time, which gives the upper bound
\begin{equation}\label{masog1}
\overline H_{1,\beta}^U(\theta,t,x) \le \theta x
\end{equation}
after taking $a\to0$ and $h\to0$. On the other hand, the particle marches deterministically to the left and to the right under the policies $\pileft$ and $\piright$, respectively (see \eqref{sagsol}), which gives the upper bound
\begin{equation}\label{masog2}
\overline H_{1,\beta}^U(\theta,t,x) \le t(\beta\mathbb{E}[V(\cdot)] - |\theta|) + \theta x
\end{equation}
by the Birkhoff ergodic theorem.

The upper bound in \eqref{masog2} is at least as good as the one in \eqref{masog1} when $\theta\ge\beta\mathbb{E}[V(\cdot)]$. In this case, we introduce a bounded and centered cocycle $G_\beta:\Omega\times\{-1,1\}\to\mathbb{R}$ (analogous to $F_{\beta,\theta}$ in Section \ref{summ1} but simpler) that satisfies
\begin{equation}\label{ozcanlar}
g_{\beta,\theta}(\omega,p) := pe^{\beta V(\omega) + \theta + G_\beta(\omega,1)} + (1-p)e^{\beta V(\omega) - \theta + G_\beta(\omega,-1)}\ge g_{\beta,\theta}(\omega,0) = e^{\beta\mathbb{E}[V(\cdot)] - \theta}
\end{equation}
for every $p\in[0,1]$. The sums $\sum G_{\beta}(T_{x_i}\omega,z_{i+1})$ over nearest-neighbor paths are uniformly sublinear in the number of steps. We use these sublinear path sums to modify the exponential expectation on the right-hand side of \eqref{homliminf} without changing the value of $\overline H_{1,\beta}^L(\theta,t,x)$. Then, it follows from the repeated application of \eqref{ozcanlar} that $\pileft$ is optimal and
\begin{equation}\label{masog3}
\overline H_{1,\beta}^L(\theta,t,x) \ge t(\beta\mathbb{E}[V(\cdot)] - \theta) + \theta x.
\end{equation}
This lower bound matches the upper bound in \eqref{masog2}
\corO{when $\theta\geq \beta \mathbb{E}[V(\cdot)]$.}

When $0<\theta<\beta\mathbb{E}[V(\cdot)]$, we introduce $\bar\beta = \bar\beta(\theta) := \frac{\theta}{\mathbb{E}[V(\cdot)]} < \beta$ and notice that
\begin{equation}\label{anago}
\overline H_{1,\beta}^L(\theta,t,x) \ge \overline H_{1,\bar\beta}^L(\theta,t,x) \ge t(\bar\beta\mathbb{E}[V(\cdot)] - \theta) + \theta x = \theta x,
\end{equation}
\corO{where the second inequality follows from}
 \eqref{masog3} since $\theta = \bar\beta\mathbb{E}[V(\cdot)]$. This lower bound matches the upper bound in \eqref{masog1}. The last two lower bounds are adapted to the $\theta<0$ case by symmetry. The $\theta=0$ case is easy. This completes the proof of Theorem \ref{thmfull}.

\subsection{Partial control: Upper bounds}\label{summ3}

The infima in the definitions of $\overline H_{\delta,\beta}^U(\theta,t,x)$ and $\overline H_{\delta,\beta}^L(\theta,t,x)$ (see \eqref{homlimsup} and \eqref{homliminf}) can be restricted to the set of bang-bang policies which take the values 
\begin{equation}\label{balerb}
\frac{1\pm\delta}{2} = \frac{e^{\pm c}}{e^c + e^{-c}} = \frac1{2}e^{\pm c - \log\cosh(c)}
\end{equation}
with the parameter $c$ introduced in \eqref{parasi}. We use \eqref{balerb} to perform a change of measure and express $\overline H_{\delta,\beta}^U(\theta,t,x)$ and $\overline H_{\delta,\beta}^L(\theta,t,x)$ in terms of expectation with respect to SSRW (see \eqref{dualformsup} and \eqref{dualforminf}). This gives an alternative formulation of our control problem where the policies are now exponential tilts denoted by $\alpha$ and taking the values $\pm c$.

In this alternative formulation, the policies $\pileft$ and $\piright$ (see \eqref{sagsol}) correspond to $\aleft$ and $\aright$ that are identically equal to $-c$ and $c$, respectively. Therefore, Theorem \ref{thmno} gives the upper bound
\begin{equation}\label{ekinoz}
\overline H_{\delta,\beta}^U(\theta,t,x) \le t\left(\min\{\Lambda_\beta(\theta - c), \Lambda_\beta(\theta + c)\} - \log\cosh(c)\right) + \theta x = t\left(\Lambda_\beta(|\theta| - c) - \log\cosh(c)\right) + \theta x.
\end{equation}

For every $h\in(0,1)$ and $\ell\in\mathbb{N}$, there is an $h$-valley of the form $[x_*-\ell,x_*+\ell-1]$ that is suitably close to the starting point of the RW, where the distance is controlled by $a>0$ as in Section \ref{summ1}. The policy $\pieL$ (see \eqref{pengol}) corresponds to $\aeL$ that is equal to $c$ at points to the left of $x_*$ and equal to $-c$ elsewhere. When $\theta = c$, the combined tilt (of $\theta$ and the control) is $2c$ at points to the left of $x_*$ and zero elsewhere, which gives a simple upper bound for $\overline H_{\delta,\beta}^U(c,t,x)$ (see \eqref{gazagel}). We dominate this upper bound using an exponential expectation involving the number of complete left excursions of a reflected RW on $[x_*-\ell,x_*+\ell-1]$ and show that
$$\overline H_{\delta,\beta}^U(c,t,x) \le t[\beta - \log\cosh(c)]^+ + cx.$$
This argument can be adapted to the $\theta = -c$ case. Finally, 
we use convexity to obtain the upper bound
\begin{equation}\label{asalk}
\overline H_{\delta,\beta}^U(\theta,t,x) \le t[\beta - \log\cosh(c)]^+ + \theta x\quad\text{for}\quad\theta\in[-c,c].
\end{equation}

Observe that, in the weak control regime ($\beta\ge\log\cosh(c)$), the 
upper bound in \eqref{asalk} is at least as good as the one in \eqref{ekinoz} 
since,
\corO{by Proposition \ref{temelsaf},}
$\Lambda_\beta(\theta \pm c) - \log\cosh(c) \ge \beta - \log\cosh(c) = [\beta - \log\cosh(c)]^+$. On the other hand, there is no such uniform (in $\theta\in[-c,c]$) comparison in the strong control regime ($\beta<\log\cosh(c)$).

\subsection{Partial control: Lower bounds}

In the alternative formulation we mentioned in Section \ref{summ3}, $\overline H_{\delta,\beta}^L(\theta,t,x)$ is expressed in terms of an exponential expectation with respect to SSRW (see \eqref{dualforminf}). Observe that the combined tilt (of $\theta$ and the control) in this expectation defines a martingale. Therefore, we can ignore its contribution at a small exponential cost by the Azuma-Hoeffding inequality, use \eqref{oakar} which is now applicable, and deduce that
$$\overline H_{\delta,\beta}^L(\theta,t,x) \ge t(\beta - \log\cosh(c)) + \theta x.$$
This lower bound matches the upper bound in \eqref{asalk} in the weak control regime ($\beta\ge\log\cosh(c)$) when $\theta\in[-c,c]$, and it also matches the upper bound in \eqref{ekinoz} when $\Lambda_\beta(|\theta| - c) = \beta$ (regardless of weak or strong control).

When $\theta\ge0$ and $\Lambda_\beta(\theta-c)>\beta$, we define
$$g_{\beta,\theta-c}(\omega,\xi) = \frac1{2}e^{\beta V(\omega) + \xi + F_{\beta,\theta-c}(\omega,1)} + \frac1{2}e^{\beta V(\omega) - \xi + F_{\beta,\theta-c}(\omega,-1)},$$
where $F_{\beta,\theta-c}$ is the corrector we mentioned in Section \ref{summ1}. Under the extra assumption that $\theta>c$, we show that
\begin{equation}\label{orform}
g_{\beta,\theta-c}(\omega,\theta+c) \ge g_{\beta,\theta-c}(\omega,\theta-c) = e^{\Lambda_\beta(\theta-c)}.
\end{equation}
Then, analogous to Sections \ref{summ1} and \ref{summ2}, we use the sublinear path sums $\sum F_{\beta,\theta-c}(T_{x_i}\omega,z_{i+1})$ to modify the exponential expectation on the right-hand side of \eqref{dualforminf}  without changing the value of $\overline H_{\delta,\beta}^L(\theta,t,x)$. By repeated application of \eqref{orform}, we deduce that $\aleft$ (i.e., $\pileft$ in the original formulation) is optimal and
\begin{equation}\label{farukg}
\overline H_{\delta,\beta}^L(\theta,t,x) \ge t(\Lambda_\beta(\theta-c) - \log\cosh(c)) + \theta x.
\end{equation}
This lower bound matches the upper bound in \eqref{ekinoz}. It is adapted to the $\theta<-c$ case by symmetry. This completes the proof of Theorem \ref{thmweak} (weak control).

It remains to obtain good lower bounds in the strong control regime ($\beta<\log\cosh(c)$) when $\theta\in(-c,c)$ and $\min\{\Lambda_\beta(\theta - c), \Lambda_\beta(\theta + c)\} > \beta$. (We know that $\Lambda_\beta(0)=\beta$.) The upper bound in \eqref{ekinoz} is at least as good as the one in \eqref{asalk} when $0<\theta<c$ and $\beta<\Lambda_\beta(\theta-c)\le\log\cosh(c)$. In this case, we show that \eqref{orform} continues to hold and it implies (by the same argument) the lower bound in \eqref{farukg}, which matches the upper bound in \eqref{ekinoz}.

When $0<\theta<c$ and $\beta<\log\cosh(c)<\Lambda_\beta(\theta-c)$, there exists a unique $\bar\theta(\beta,c)\in(0,c)$ such that
$$\Lambda_\beta(\bar\theta(\beta,c)-c) = \log\cosh(c)$$
and a unique $\bar\beta = \bar\beta(\theta,c)<\beta$ such that $\theta = \bar\theta(\bar\beta,c)$.
Analogous to \eqref{anago} in Section \ref{summ2},
$$\overline H_{\delta,\beta}^L(\theta,t,x) \ge \overline H_{\delta,\bar\beta}^L(\theta,t,x)\ge t(\Lambda_{\bar\beta}(\bar\theta(\bar\beta,c) - c) - \log\cosh(c)) + \theta x = \theta x$$
by \eqref{farukg} since $\bar\beta < \Lambda_{\bar\beta}(\theta - c) = \log\cosh(c)$. This lower bound matches the upper bound in \eqref{asalk}. The last two lower bounds are adapted to the $-c<\theta<0$ case by symmetry. The $\theta=0$ case is easy. This completes the proof of Theorem \ref{thmstrong} (strong control).

\section{No control}\label{zerocontrol}

\subsection{The tilted free energy}\label{existtilt}

In this section, we provide a self-contained proof of Theorem \ref{thmno} (see Section \ref{subseczeroc} for references to the literature on the existence of the tilted free energy). While doing so, we obtain some intermediate results which will be central to the proofs of Theorems \ref{thmfull}, \ref{thmweak}, \ref{thmstrong} and \ref{unifhom2}.

We start by recalling an elementary result regarding SSRW.

\begin{lemma}\label{appasakla}
For every $\ell\in\mathbb{N}$,
$$\lim_{n\to\infty}\frac1{n}\log P_0(X_i\in[-\ell,\ell-1]\ \text{for every}\ i\in[0,n]) = \log\cos(\pi/(2\ell + 1)).$$
\end{lemma}

\begin{proof}
	This follows immediately from the eigenvalues and eigenvectors of the adjacency matrix of SSRW on $[-\ell,\ell-1]$ with absorbing boundary conditions (see \cite[p.\ 239]{Spi1976}).
\end{proof}

Assume that \eqref{ass_wlog} and \eqref{ass_pan} hold. For every $\beta>0$, $\theta\in\mathbb{R}$, $t>0$ and $x\in\mathbb{R}$, let 
\begin{align}
\Lambda_\beta^L(\theta,t,x) &= \liminf_{\ep\to0}\ep\log E_{[\epin x]}\left[e^{\beta\sum_{i=0}^{[\epin t]-1}V(T_{X_i}\omega) + \theta X_{[\epin t]}}\right]\quad \text{and}\label{kosk1}\\
\Lambda_\beta^U(\theta,t,x) &= \limsup_{\ep\to0}\ep\log E_{[\epin x]}\left[e^{\beta\sum_{i=0}^{[\epin t]-1}V(T_{X_i}\omega) + \theta X_{[\epin t]}}\right].\label{kosk2}
\end{align}
Strictly speaking, we should write $\Lambda_\beta^L(\theta,t,x,\omega)$ and $\Lambda_\beta^U(\theta,t,x,\omega)$ to indicate the dependence on $\omega$, too. However, it is clear from the ellipticity of SSRW that $\Lambda_\beta^L(\theta,t,x,\omega) = \Lambda_\beta^L(\theta,t,x,T_1\omega)$, and therefore $\Lambda_\beta^L(\theta,t,x)$ is $\mathbb{P}$-a.s.\ constant by the ergodicity assumption. The same reasoning applies to $\Lambda_\beta^U(\theta,t,x)$. For the purpose of proving the existence of the tilted free energy $\Lambda_\beta(\theta)$ (see \eqref{nolimit}), it suffices to take $t=1$ and $x=0$. The latter applies to the following lemma, too.

\begin{lemma}\label{gayettehos}
For every $h\in(0,1)$, $\ell\in\mathbb{N}$, $a>0$, $B>0$ and $\mathbb{P}$-a.e.\ $\omega$, there exists an $n_0 = n_0(\omega,h,\ell,a,B)$ such that the interval $[n(x-a),n(x+a)]$ contains an $h$-valley (resp.\ $h$-hill) of the form $[x_*-\ell,x_*+\ell -1]$ (resp.\ $[x^*-\ell,x^*+\ell -1]$) for every $x\in[-B,B]$ and $n\ge n_0$.
\end{lemma}

\begin{proof}
  \corOO{Without loss of generality, we may and will assume that $a/B\leq 1$.}
  For every $h\in(0,1)$, $k,\ell\in\mathbb{N}$ and $\mathbb{P}$-a.e.\ $\omega$, the number of $h$-valleys of length $2\ell - 1$ contained in the interval $[0,k]$ (resp.\ the interval $[-k,0]$) is $kp_{h,\ell} + o(k)$
by the Birkhoff ergodic theorem, where
$$p_{h,\ell} := \mathbb{P}(\text{$[0,2\ell-1]$ is an $h$-valley})>0$$
by \eqref{ass_pan}. Therefore, for 
\corOO{$j\in[-\frac{B}{a},\frac{B}{a}-1]\cap\mathbb{Z}$},  
  the number of $h$-valleys of length $2\ell-1$ contained in the interval 
  \corOO{$[0,nja]$ is $n|j|ap_{h,\ell}(1\pm  a/4B)$, 
  for all $n\geq n_1(\omega,h,\ell,a,B)$. 
    It follows that for such $n$, the number of $h$-valleys of
    length $2\ell-1$ contained in $[nja,n(j+1)a]$ is 
    at least $\frac12 nap_{h,\ell}$.
  In particular, it is positive for $n\geq n_0(\omega,h,\ell,a,B)$.} 
  For every $x\in[-B,B]$, the interval $[n(x-a),n(x+a)]$ contains 
  $[nja,n(j+1)a]$ for at least one such $j$, and the desired result follows. The same argument applies to $h$-hills of length $2\ell-1$.
\end{proof}

\begin{lemma}\label{kendiyag}
$\Lambda_\beta^L(\theta,t,x) \ge t\beta + \theta x$ for every $\beta>0$, $\theta\in\mathbb{R}$, $t>0$ and $x\in\mathbb{R}$. 
\end{lemma}

\begin{proof}
For every $h\in(0,1)$, $\ell\in\mathbb{N}$, $t>0$, $x\in\mathbb{R}$, $a\in(0,t)$, $\mathbb{P}$-a.e.\ $\omega$ and sufficiently small $\ep>0$, Lemma \ref{gayettehos} implies the existence of an $h$-hill of the form $[x^* - \ell,x^* + \ell - 1]$ that is contained in the interval $[\epin(x-a),\epin(x+a)]$. Let $A_{[\epin t]}(x^*,h,\ell)$ be the event that the particle marches deterministically from $[\epin x]$
to $x^*$ and then spends the rest of the $[\epin t]$ units of time in this $h$-hill. Restricting on this event and applying Lemma \ref{appasakla}, we get the following lower bound:
\begin{align*}
E_{[\epin x]}\left[e^{\beta\sum_{i=0}^{[\epin t]-1}V(T_{X_i}\omega) + \theta X_{[\epin t]}}\right] &\ge E_{[\epin x]}\left[e^{\beta\sum_{i=0}^{[\epin t]-1}V(T_{X_i}\omega) + \theta X_{[\epin t]}}\one_{A_{[\epin t]}(x^*,h,\ell)}\right]\\
&\ge e^{([\epin t] - \epin a)\beta h + \theta[\epin x] - |\theta|\epin a}P_{[\epin x]}(A_{[\epin t]}(x^*,h,\ell))\\
&\ge e^{([\epin t] - \epin a)\beta h + \theta[\epin x] - (|\theta| + \log 2)\epin a + [\epin t]\log\cos(\pi/(2\ell + 1)) + o(\epin t)}.
\end{align*}
The desired result is obtained by first taking $\ep\log$ of both sides, then sending $\ep\to0$, and finally taking $a\to0$, $h\to 1$ and $\ell\to\infty$.
\end{proof}

Let $\tau_k = \inf\{i\ge0: X_i = k\}$ denote the first time the particle is at $k\in\mathbb{Z}$. For every $\lambda > \beta$, define
\begin{equation}
\begin{aligned}\label{pregaztap}
F_{\beta,\theta}^\lambda(\omega,1) &= - \log E_0\left[e^{\beta\sum_{i=0}^{\tau_1-1}V(T_{X_i}\omega) + \theta X_{\tau_1} - \lambda\tau_1}\one_{\{\tau_1<\infty\}}\right]\\
&= - \log E_0\left[e^{\beta\sum_{i=0}^{\tau_1-1}V(T_{X_i}\omega) - \lambda\tau_1}\one_{\{\tau_1<\infty\}}\right] - \theta.
\end{aligned}
\end{equation}
Note that 
\begin{equation}
\begin{aligned}\label{fbounds}
F_{\beta,\theta}^\lambda(\omega,1) & \ge - \log E_0\left[e^{(\beta - \lambda)\tau_1}\one_{\{\tau_1<\infty\}}\right] - \theta > (\lambda - \beta) - \theta\quad\text{and}\\
F_{\beta,\theta}^\lambda(\omega,1) & \le - \log E_0\left[e^{-\lambda\tau_1}\one_{\{\tau_1<\infty\}}\right] - \theta < \infty.
\end{aligned}
\end{equation}
Set $F_{\beta,\theta}^\lambda(\omega,-1) = - F_{\beta,\theta}^\lambda(T_{-1}\omega,1)$. 
Then, decomposing the expectation corresponding to $e^{-F_{\beta,\theta}^\lambda(\omega,1)}$ with respect to the first step of the RW, we see that
\begin{equation}\label{peeling}
e^\lambda  = \frac1{2}e^{\beta V(\omega) + \theta + F_{\beta,\theta}^\lambda(\omega,1)} + \frac1{2}e^{\beta V(\omega) - \theta + F_{\beta,\theta}^\lambda(\omega,-1)}
\end{equation}
for every $\omega\in\Omega$.

\begin{lemma}\label{decomp}
If $\theta\ge0$ and $\beta < \lambda < t^{-1}\left(\Lambda_\beta^U(\theta,t,x) - \theta x\right)$ for some $t>0$ and $x\in\mathbb{R}$, then $$\mathbb{E}[F_{\beta,\theta}^\lambda(\cdot,1)] \le \lambda  - t^{-1}\left(\Lambda_\beta^U(\theta,t,x) - \theta x\right) < 0.$$
\end{lemma}

\begin{proof}
  \corOO{Fix} $\theta\ge0$ and 
  $\beta < \lambda < t^{-1}\left(\Lambda_\beta^U(\theta,t,x) - 
  \theta x\right)$ for some $t>0$ and $x\in\mathbb{R}$. 
  \corOO{Then, for $\mathbb{P}$-a.e.\ $\omega$, there exists 
  a subsequence $\epsilon_k\to 0$ and an error $o(\epsilon_k^{-1})$,
  both possibly depending on
$\beta$, $\theta$, $t$, $x$ and $\omega$, so that}
\begin{equation}
  \label{eq-cancun1}
	\corOO{
	  e^{\epin_k(\Lambda_\beta^U(\theta,t,x) - 
	  t\lambda - \theta x) + o(\epin_k)} = 
	  E_{[\epin_k x]}\left[e^{\beta\sum_{i=0}^{[\epin_k t]-1}V(T_{X_i}\omega) + \theta (X_{[\epin_k t]} - [\epin_k x]) - [\epin_k t]\lambda}\right].
	}
	\end{equation}
	\corOO{On the other hand, for any $\ep>0$,}
	\begin{align}
	  \label{eq-cancun2}
	  &E_{[\epin x]}\left[e^{\beta\sum_{i=0}^{[\epin t]-1}V(T_{X_i}\omega) + \theta (X_{[\epin t]} - [\epin x]) - [\epin t]\lambda}\right]\nonumber\\
	&= \sum_{k=0}^{[\epin t]-1}E_0\left[e^{\sum_{i=0}^{[\epin t]-1}[\beta V(T_{[\epin x] + X_i}\omega) - \lambda] + \theta X_{[\epin t]}}\one_{\{\tau_{k} < [\epin t] \le \tau_{k + 1}\}}\right]\nonumber\\
	&\le \sum_{k=0}^{[\epin t]-1}E_0\left[e^{\sum_{i=0}^{\tau_k-1}[\beta V(T_{[\epin x] + X_i}\omega) - \lambda] + \theta(k+1)}\one_{\{\tau_{k} < [\epin t] \le \tau_{k + 1}\}}\right]\nonumber\\
	&\le \sum_{k=0}^{[\epin t]-1}e^{\theta}E_0\left[e^{\sum_{i=0}^{\tau_k-1}[\beta V(T_{[\epin x] + X_i}\omega) - \lambda] + \theta X_{\tau_{k}}}\one_{\{\tau_{k} <\infty\}}\right]\nonumber\\
	&= \sum_{k=0}^{[\epin t]-1}e^{\theta - \sum_{j=[\epin x]}^{[\epin x]+k-1}F_{\beta,\theta}^\lambda(T_j\omega,1)} .
	\end{align}
	\corOO{We now claim that for any $\eta>0$ there is
	  an $\epsilon_0=\epsilon_0(\omega,\beta,\theta,\lambda,\eta)$ so that
	if $\ep\le\ep_0$ then for all $0\leq k\leq [\epin t]-1$,}
	\begin{equation}
	  \label{eq-cancun3}
	\corOO{\left|\sum_{j=[\epin x]}^{[\epin x]+k-1}F_{\beta,\theta}^\lambda(T_j\omega,1) - k\mathbb{E}[F_{\beta,\theta}^\lambda(\cdot,1)]\right|\le \epin\eta 
      (1+t + 2|x|).}
      \end{equation}
      \corOO{Combining \eqref{eq-cancun1}--\eqref{eq-cancun3}, we conclude that
      $0 < \Lambda_\beta^U(\theta,t,x) - t\lambda - \theta x \le - t\mathbb{E}[F_{\beta,\theta}^\lambda(\cdot,1)]$, which completes the proof of the lemma.}

      \corOO{It remains to prove \eqref{eq-cancun3}.}
	 Note that
	 the Birkhoff ergodic theorem gives an 
	 $n_0 = n_0(\omega,\beta,\theta,\lambda,\eta)$ such that 
	 $n\ge n_0$ implies 
	 $$\left|\sum_{j=0}^{n-1}F_{\beta,\theta}^\lambda(T_j\omega,1) - 
	 n\mathbb{E}[F_{\beta,\theta}^\lambda(\cdot,1)]\right|\le n\eta.$$
	 Hence, for $0\le k\le [\epin t]-1$,
\begin{equation}\label{sumdiff}
\left|\sum_{j=[\epin x]}^{[\epin x]+k-1}F_{\beta,\theta}^\lambda(T_j\omega,1) - k\mathbb{E}[F_{\beta,\theta}^\lambda(\cdot,1)]\right|\le \epin\eta (t + 2|x|)
\end{equation}
whenever $|[\epin x]| \ge n_0$ and $|[\epin x]+k| \ge n_0$. 
Otherwise, we can shift the indices of the sum in 
\eqref{sumdiff} by $2n_0$, recall 
\eqref{fbounds} and use the triangle inequality to deduce that
\corOO{the left side of \eqref{sumdiff} is bounded by 
$\epin\eta (t + 2|x|)+4n_0\|F^\lambda_{\beta,\theta}(\cdot,1)\|_\infty$.
Choosing $\epsilon_0<\eta/(4 n_0\|F^\lambda_{\beta,\theta}(\cdot,1)\|_\infty)$
then gives \eqref{eq-cancun3}. }
\end{proof}

\begin{lemma}\label{lemincrease}
If $\theta\ge0$, then the map $\lambda\mapsto \mathbb{E}[F_{\beta,\theta}^\lambda(\cdot,1)]$ is continuous and strictly increasing for $\lambda > \beta$. Moreover,
$$\lim_{\lambda\to\infty}\mathbb{E}[F_{\beta,\theta}^\lambda(\cdot,1)] = \infty.$$
\end{lemma}

\begin{proof}
Note that $F_{\beta,\theta}^{\lambda + \Delta\lambda}(\omega,1) \ge F_{\beta,\theta}^\lambda(\omega,1) + \Delta\lambda$, since $\tau_1\ge1$. The rest follows from the uniform (in $\omega$) bounds in \eqref{fbounds} and the dominated convergence theorem.
\end{proof}


\begin{proof}[Proof of Theorem \ref{thmno}]
Assume without loss of generality that $\theta\ge0$.

If $\Lambda_\beta^U(\theta,t,x) = t\beta + \theta x$ for every $t>0$ and $x\in\mathbb{R}$, then Lemma \ref{kendiyag} gives
\begin{equation}\label{begbit}
\Lambda_\beta^L(\theta,t,x) = \Lambda_\beta^U(\theta,t,x) = t\beta + \theta x.
\end{equation}

If $\Lambda_\beta^U(\theta,t',x') > t'\beta + \theta x'$ for some $t'>0$ and $x'\in\mathbb{R}$, then 
Lemmas \ref{decomp} and \ref{lemincrease} (and the intermediate value theorem) imply the existence of a unique $\lambda \ge (t')^{-1}\left(\Lambda_\beta^U(\theta,t',x') - \theta x'\right)$ such that $\mathbb{E}[F_{\beta,\theta}^\lambda(\cdot,1)] = 0$. Then, $F_{\beta,\theta}^\lambda$ is a bounded and centered cocycle (see Definition \ref{cencocdef} in Appendix \ref{app_cocycle}).
Therefore, for every $t>0$, $x\in\mathbb{R}$ and $\mathbb{P}$-a.e.\ $\omega$,
\begin{align*}
&E_{[\epin x]}\left[e^{\beta\sum_{i=0}^{[\epin t]-1}V(T_{X_i}\omega) + \theta X_{[\epin t]}}\right]\\
&\ \;  = E_0\left[e^{\sum_{i=0}^{[\epin t]-1}[\beta V(T_{[\epin x] + X_i}\omega) + \theta Z_{i+1} + F_{\beta,\theta}^\lambda(T_{[\epin x] + X_i}\omega,Z_{i+1})]}\right]e^{\theta[\epin x] + o(\epin t)}\\
&\ \; = e^{[\epin t]\lambda + \theta[\epin x] + o(\epin t)} = e^{\epin(t\lambda + \theta x) + o(\epin t)}.
\end{align*}
Here, the first equality follows from the uniformly sublinear (in $[\epin t]$) growth of sums (over nearest-neighbor paths of length $[\epin t]$) of bounded and centered cocycles (see Lemma \ref{cencoclem}) and the last equality is obtained by the repeated application of (\ref{peeling}). Taking $\ep\log$ of both sides and sending $\ep\to0$, we conclude that \begin{equation}\label{gonbit}
\Lambda_\beta^L(\theta,t,x) = \Lambda_\beta^U(\theta,t,x) = t\lambda + \theta x.
\end{equation}

The existence of the limit in \eqref{homlimit} and the validity of the identity in \eqref{simdigel} follow immediately from \eqref{begbit} and \eqref{gonbit}. Finally, setting $t=1$ and $x=0$, we deduce \eqref{nolimit}.
\end{proof}

\subsection{The corrector and an implicit formula}\label{korgozu}

When $\theta > 0$ and $\Lambda_\beta(\theta) > \beta$, we will henceforth write $F_{\beta,\theta} = F_{\beta,\theta}^{\Lambda_\beta(\theta)}$ to simplify the notation in the previous section. We extend this definition to the $\theta < 0$ case and recapitulate it as follows:
\begin{equation}
\begin{aligned}\label{gaztap}
F_{\beta,\theta}(\omega,1) &= - \log E_0\left[e^{\beta\sum_{i=0}^{\tau_{1}-1}V(T_{X_i}\omega) - \Lambda_\beta(\theta)\tau_{1}}\one_{\{\tau_{1}<\infty\}}\right] - \theta\quad\text{and}\\ F_{\beta,\theta}(\omega,-1) &= -F_{\beta,\theta}(T_{-1}\omega,1)\quad\text{if $\theta > 0$ and $\Lambda_\beta(\theta) > \beta$;}\\
F_{\beta,\theta}(\omega,-1) &= - \log E_0\left[e^{\beta\sum_{i=0}^{\tau_{-1}-1}V(T_{X_i}\omega) - \Lambda_\beta(\theta)\tau_{-1}}\one_{\{\tau_{-1}<\infty\}}\right] + \theta\quad\text{and}\\
F_{\beta,\theta}(\omega,1) &= -F_{\beta,\theta}(T_{1}\omega,-1)\quad\text{if $\theta < 0$ and $\Lambda_\beta(\theta) > \beta$.}
\end{aligned}
\end{equation}
Note that this definition leaves out $\theta = 0$ because $\Lambda_\beta(0) = \beta$ (see Proposition \ref{temelsaf}(c)). We record the following results for future reference.

\begin{proposition}\label{cikmazde}
	Assume \eqref{ass_wlog}, \eqref{ass_pan},
	\corO{and that $\theta$ is such that
	$\Lambda_\beta(\theta) > \beta$.} Then $F_{\beta,\theta}:\Omega\times\{-1,1\}\to\mathbb{R}$, defined in \eqref{gaztap}, satisfies
	\begin{equation}\label{melburn}
	\mathbb{E}[F_{\beta,\theta}(\cdot,\pm 1)] = 0,
	\end{equation}
	i.e., it is a centered cocycle (see Definition \ref{cencocdef} in Appendix \ref{app_cocycle}). Moreover, for every $\omega\in\Omega$,
	\begin{align}
	0 < \Lambda_\beta(\theta) - \beta &< |\theta| + F_{\beta,\theta}(\omega,\text{sgn}(\theta)) \le - \log E_0\left[e^{-\Lambda_\beta(\theta)\tau_1}\one_{\{\tau_1<\infty\}}\right] < \infty\quad\text{and}\label{fboundsrec}\\
    e^{\Lambda_\beta(\theta)} &= \frac1{2}e^{\beta V(\omega) + \theta + F_{\beta,\theta}(\omega,1)} + \frac1{2}e^{\beta V(\omega) - \theta + F_{\beta,\theta}(\omega,-1)}.\label{menzilde}
    \end{align}
\end{proposition}
\begin{proof}
\corO{  The equality \eqref{melburn} follows from the definition of
$\Lambda_\beta(\theta)$, building on the proof of Theorem \ref{thmno}.}
When $\theta > 0$, the desired results \eqref{fboundsrec} and
\eqref{menzilde} have been shown in \eqref{fbounds} and \eqref{peeling}, respectively.
When $\theta < 0$, the proofs are identical since the law of the underlying SSRW is symmetric.
\end{proof}

In light of the exact equality in \eqref{menzilde}, $F_{\beta,\theta}$ is referred to as the corrector. We will say more about this choice of terminology in Appendix \ref{app_varfor} (see Remark \ref{correctorname}) where we present and analyze two variational formulas for $\Lambda_\beta(\theta)$. The following result gives an implicit (non-variational) formula for $\Lambda_\beta(\theta)$.

\begin{proposition}\label{propimpfor}
	Assume \eqref{ass_wlog} and \eqref{ass_pan}. Then, $\Lambda_\beta(\theta) \ge \beta$ for every $\theta\in\mathbb{R}$. Moreover,
	\begin{equation}\label{impfor}
	\mathbb{E}\left[\log E_0\left[e^{\beta\sum_{i=0}^{\tau_1-1}V(T_{X_i}\omega) - \Lambda_\beta(\theta)\tau_1}\one_{\{\tau_1<\infty\}}\right]\right] + |\theta| = 0
	\end{equation}
	whenever $\Lambda_\beta(\theta) > \beta$.
\end{proposition}

\begin{proof}
	We already know from Lemma \ref{kendiyag} (with $t=1$ and $x=0$) that $\Lambda_\beta(\theta) \ge \beta$. The symmetry of the law of SSRW implies that $\Lambda_\beta(\theta)$ is even in $\theta$. (We will list various properties of the tilted free energy in Proposition \ref{temelsaf} below.) Therefore, \eqref{impfor} follows from \eqref{gaztap} and \eqref{melburn} 
	whenever $\Lambda_\beta(\theta) > \beta$.
\end{proof}

\subsection{Some properties of the tilted free energy}


\begin{proposition}\label{temelsaf}
	Assume \eqref{ass_wlog} and \eqref{ass_pan}. Then, the following hold.
	\begin{itemize}
		\item [(a)] $\Lambda_\beta(\theta)$ is increasing in $\beta$, and even and convex in $\theta$.
		\item [(b)] $\Lambda_\beta(\theta) \ge \max\left\{\beta,\beta\mathbb{E}[V(\cdot)] + \log\cosh(\theta)\right\}$ for every $\theta\in\mathbb{R}$. 
		\item [(c)] If $|\theta| \le\beta(1 - \mathbb{E}[V(\cdot)])$, then $\Lambda_\beta(\theta) = \beta$. Hence, the set $\{\theta\in\mathbb{R}:\,\Lambda_\beta(\theta) = \beta\}$ is a symmetric and closed interval with nonempty interior.
		\item [(d)] $\Lambda_\beta(\theta) - \log\cosh(\theta) \to \beta\mathbb{E}[V(\cdot)]$ as $|\theta|\to\infty$.
		\item [(e)] The map $\theta\mapsto\Lambda_\beta(\theta)$ is continuously differentiable on the complement of $\{\theta\in\mathbb{R}:\,\Lambda_\beta(\theta) = \beta\}$.
	\end{itemize}
\end{proposition}

\begin{proof}
	\begin{itemize}
		\item [(a)] These three properties follow from $V(\cdot)\ge0$, the symmetry of the law of SSRW and a standard application of H\"older's inequality, respectively.
		\item [(b)] Consider the nearest-neighbor RW with probability $p(\theta) = e^\theta/(e^\theta + e^{-\theta})$ of jumping to the right.
		It induces a probability measure $\hat P_0^\theta$ on paths starting at $0$. Let $\hat E_0^\theta$ denote expectation under $\hat P_0^\theta$. Note that
		$$\mathbb{E}[p(\theta)f(T_1\cdot) + (1-p(\theta))f(T_{-1}\cdot)] = p(\theta)\mathbb{E}[f(T_1\cdot)] + (1-p(\theta))\mathbb{E}[f(T_{-1}\cdot)] = \mathbb{E}[f(\cdot)]$$
		for every bounded and $\mathcal{F}$-measurable function $f:\Omega\to\mathbb{R}$. In other words, under $\hat P_0^\theta$, 
		the probability measure $\mathbb{P}$ is invariant for 
		the so-called environment Markov chain 
		$(T_{X_i}\omega)_{i\ge0}$, and hence 
		ergodic 
		\corOO{(with respect to temporal shifts)}
		by Kozlov's lemma (see \cite{Koz1985} for details). Therefore,
		\begin{align*}
		\Lambda_\beta(\theta) &= \lim_{n\to\infty}\frac1{n}\log E_0\left[e^{\beta\sum_{i=0}^{n-1}V(T_{X_i}\omega) + \theta X_n}\right] = \lim_{n\to\infty}\frac1{n}\log \hat E_0^\theta\left[e^{\beta\sum_{i=0}^{n-1}V(T_{X_i}\omega)}\right] + \log\cosh(\theta)\\
		&\ge \lim_{n\to\infty}\beta\hat E_0^\theta\left[\frac1{n}\sum_{i=0}^{n-1}V(T_{X_i}\omega)\right] + \log\cosh(\theta) = \beta\mathbb{E}[V(\cdot)] + \log\cosh(\theta)
		\end{align*}
		by Jensen's inequality, the Birkhoff ergodic theorem and the bounded convergence theorem. Recalling the first part of Proposition \ref{propimpfor}, we get the desired lower bound.
		\item [(c)] If $|\theta| \le \beta(1 - \mathbb{E}[V(\cdot)])$, 
		  then, 
		  \corOO{with the $o(\cdot)$ notation denoting 
		    error terms that may depend
		on $\omega$,}
		$$e^{n\Lambda_\beta(\theta) + o(n)} = E_0\left[e^{\beta\sum_{i=0}^{n-1} V(T_{X_i}\omega) + \theta X_n}\right] \le E_0\left[e^{n\beta + (|\theta| - \beta(1 - \mathbb{E}[V(\cdot)]))|X_n| + o(|X_n|)}\right] \le e^{n\beta + o(n)}$$
		for $\mathbb{P}$-a.e.\ $\omega$ by the Birkhoff ergodic theorem and the observation that the particle visits each $x$ between $0$ and $X_n$ at least once. Therefore, \corOO{$\Lambda_\beta(\theta)\le\beta$ and one concludes
		by appealing to} part (b).
		\item [(d)] Similar to part (c), if $|\theta| \ge \beta(1 - \mathbb{E}[V(\cdot)])$, then
		$$e^{n\Lambda_\beta(\theta) + o(n)} \le E_0\left[e^{n\beta + (|\theta| - \beta(1 - \mathbb{E}[V(\cdot)]))|X_n| + o(|X_n|)}\right] \le e^{n\beta + n\log\cosh(|\theta| - \beta(1 - \mathbb{E}[V(\cdot)])) + o(n)}$$
		for $\mathbb{P}$-a.e.\ $\omega$. Therefore, $$\beta\mathbb{E}[V(\cdot)] \le \Lambda_\beta(\theta) - \log\cosh(\theta) \le \beta + \log\cosh(|\theta| - \beta(1 - \mathbb{E}[V(\cdot)])) - \log\cosh(\theta)$$
		by part (b), and the desired result follows.
		\item [(e)] Recall from \eqref{pregaztap} that
		$$F_{\beta,\theta}^\lambda(\omega,1) = - \log E_0\left[e^{\beta\sum_{i=0}^{\tau_1-1}V(T_{X_i}\omega) - \lambda\tau_1}\one_{\{\tau_1<\infty\}}\right] - \theta$$
		for every $\omega\in\Omega$, $\theta>0$ and $\lambda>\beta$. Since $0\le V(\cdot)\le 1$, it follows from an application of the dominated convergence theorem (DCT) that the map $\lambda\mapsto  F_{\beta,\theta}^\lambda(\omega,1)$ is differentiable. By a second application of the DCT, we deduce that the map $\lambda\mapsto  \mathbb{E}[F_{\beta,\theta}^\lambda(\cdot,1)]$ is differentiable and
		\begin{equation}\label{ustdest}
		\frac{\partial}{\partial\lambda}\mathbb{E}[F_{\beta,\theta}^\lambda(\cdot,1)] = \mathbb{E}\left[\frac{E_0\left[\tau_1e^{\beta\sum_{i=0}^{\tau_1-1}V(T_{X_i}\omega) - \lambda\tau_1}\one_{\{\tau_1<\infty\}}\right]}{E_0\left[e^{\beta\sum_{i=0}^{\tau_1-1}V(T_{X_i}\omega) - \lambda\tau_1}\one_{\{\tau_1<\infty\}}\right]}\right]>0.
		\end{equation}
		Resorting to the DCT for a third time, we see that $\lambda\mapsto  \mathbb{E}[F_{\beta,\theta}^\lambda(\cdot,1)]$ is in fact continuously differentiable.
		The map $\theta\mapsto\mathbb{E}[F_{\beta,\theta}^\lambda(\cdot,1)]$ is linear, and hence continuously differentiable, too. 
		Recall from \eqref{melburn} that
		\begin{equation}
		  \label{eq-star14}
		  \mathbb{E}[F_{\beta,\theta}^{\Lambda_\beta(\theta)}(\cdot,1)] = \mathbb{E}[F_{\beta,\theta}(\cdot,1)] = 0
		\end{equation}
		whenever $\theta>0$ and $\Lambda_\beta(\theta) > \beta$.
		Thus, by the implicit function theorem, the map $\theta\mapsto\Lambda_\beta(\theta)$ is continuously differentiable on the set $\{\theta\in\mathbb{R}:\,\theta>0\ \text{and}\ \Lambda_\beta(\theta)>\beta\}$. Since $\Lambda_\beta(-\theta) = \Lambda_\beta(\theta)$ and $\Lambda_\beta(0) = \beta$ by parts (a) and (c), this concludes the proof.\qedhere
	\end{itemize}
\end{proof}

	Proposition \ref{temelsaf} does not answer the question of whether $\theta\mapsto\Lambda_\beta(\theta)$ is differentiable at the endpoints of the interval $\{\theta\in\mathbb{R}:\,\Lambda_\beta(\theta) = \beta\}$. We provide a negative answer to this question in Appendix \ref{app_isimyok} under a very mild additional assumption on the potential (see Theorem \ref{nondifgen} for the precise statement). This nondifferentiability is reflected in the sketches in Figure \ref{matrixfigure}, but it is not actually used anywhere in the paper.

%
%

\section{Full control}\label{ayirmayir}

For every $\delta\in(0,1]$, $\beta>0$, $\theta\in\mathbb{R}$, $t\ge0$ and $x\in\mathbb{R}$, let 
\begin{align}
\overline H_{\delta,\beta}^U(\theta,t,x) &= \limsup_{\ep\to0}\inf_{\pi\in\mathcal{P}_{[\epin t]}(\delta)}\ep\log E_{[\epin x]}^{\pi,\omega}\left[e^{\beta\sum_{i=0}^{[\epin t]-1}V(T_{X_i}\omega) + \theta X_{[\epin t]}}\right]\quad \text{and}\label{homlimsup}\\
\overline H_{\delta,\beta}^L(\theta,t,x) &= \liminf_{\ep\to0}\inf_{\pi\in\mathcal{P}_{[\epin t]}(\delta)}\ep\log E_{[\epin x]}^{\pi,\omega}\left[e^{\beta\sum_{i=0}^{[\epin t]-1}V(T_{X_i}\omega) + \theta X_{[\epin t]}}\right].\label{homliminf}
\end{align}

In this section, we assume that \eqref{ass_wlog} and \eqref{ass_pan} hold, take $\delta = 1$, provide matching upper and lower bounds for \eqref{homlimsup} and \eqref{homliminf}, respectively, and prove Theorem \ref{thmfull}. In fact, we go beyond Theorem \ref{thmfull} and obtain error bounds for the limit in \eqref{homlimit} 
which will be used in the proof of Theorem \ref{unifhom2}.

\subsection{Upper bounds}\label{apirband}

For every $h\in(0,1)$, $t\ge0$, $x\in\mathbb{R}$, $a>0$, $\mathbb{P}$-a.e.\ $\omega$ and sufficiently small $\ep>0$, Lemma \ref{gayettehos} implies the existence of an $h$-valley of the form $[x_* - 1,x_*]$ that is contained in the interval $[\epin(x-a),\epin(x+a)]$. Consider the policy $\pitwo$ (given in \eqref{pengol}) with this specific choice of $x_*$ (see Remark \ref{dundundur}). Under this policy, the particle marches deterministically to $x_*$ and is then confined to the $h$-valley $[x_*-1,x_*]$ for the rest of the $[\epin t]$ units of time (if it gets to $x_*$). This gives the following bound:
\begin{equation}
\begin{aligned}\label{mazharfu}
\ep\log E_{[\epin x]}^{\pitwo,\omega}\left[e^{\beta\sum_{i=0}^{[\epin t]-1}V(T_{X_i}\omega) + \theta X_{[\epin t]}}\right] &\le \ep[\epin t]\beta h + \ep[\epin a](\beta + |\theta|) + \ep\theta [\epin x]\\
&\le t\beta h + a(\beta + |\theta|) + \theta x + \ep|\theta|. 
\end{aligned}
\end{equation}
Sending $\ep\to0$, $h\to0$ and $a\to0$ \corOO{in this order}, we deduce that
\begin{equation}\label{UBfull1}
\overline H_{1,\beta}^U(\theta,t,x) \le \theta x.
\end{equation}

\begin{remark}\label{dundundur}
	In Section \ref{asop}, we introduced the RW policy $\pieL$ using an $h$-valley of the form $[x_*-\ell,x_*+\ell-1]$. 
	When the walk starts at the origin (e.g., in \eqref{homlimsup} with $x=0$), we can work with a fixed $x_* = x_*(\omega,h,\ell)$ for all sufficiently small $\ep>0$. However, when the walk starts at $[\epin x]$ with some $x\ne0$, we need to take $x_* = x_*(\omega,h,\ell,x,a,\ep)$ as in Lemma \ref{gayettehos}. In particular, the policy $\pieL$ depends on $\ep$ in the latter case.
\end{remark}

When $\theta\ge 0$, consider the policy $\pileft$ (given in \eqref{sagsol}) under which the particle marches deterministically to the left for $[\epin t]$ units of time. 
\begin{equation}
\begin{aligned}\label{fuatoz}
&\ep\log E_{[\epin x]}^{\pileft,\omega}\left[e^{\beta\sum_{i=0}^{[\epin t]-1}V(T_{X_i}\omega) + \theta X_{[\epin t]}}\right] = \ep\sum_{i=0}^{[\epin t] - 1}(\beta V(T_{[\epin x] - i}\omega) - \theta) +  \ep\theta[\epin x]\\
&\qquad = \ep [\epin t](\beta\mathbb{E}[V(\cdot)] - \theta) + \ep\theta[\epin x] + \ep\sum_{i=0}^{[\epin t] - 1}(\beta V(T_{[\epin x] - i}\omega) - \beta\mathbb{E}[V(\cdot)])\\
&\qquad \le t(\beta\mathbb{E}[V(\cdot)] - \theta) + \theta x + \ep\sum_{i=0}^{[\epin t] - 1}(\beta V(T_{[\epin x] - i}\omega) - \beta\mathbb{E}[V(\cdot)]) + \ep(\beta + 2|\theta|).
\end{aligned}
\end{equation}
By the Birkhoff ergodic theorem, we deduce the following bound: for $\mathbb{P}$-a.e.\ $\omega$,
\begin{equation}\label{UBfull2}
\overline H_{1,\beta}^U(\theta,t,x) \le \limsup_{\ep\to0}\ep\log E_{[\epin x]}^{\pileft,\omega}\left[e^{\beta\sum_{i=0}^{[\epin t]-1}V(T_{X_i}\omega) + \theta X_{[\epin t]}}\right] = t(\beta\mathbb{E}[V(\cdot)] - \theta) + \theta x.
\end{equation}

\subsection{Lower bounds when $\theta\ge 0$}\label{lovirband}

\subsubsection{Lower bound when $\theta\ge\beta\mathbb{E}[V(\cdot)]$}

Define $G_\beta:\Omega\times\{-1,1\}\to\mathbb{R}$ by
\begin{equation}\label{ptesi}
G_\beta(\omega,-1) = -\beta V(\omega) + \beta\mathbb{E}[V(\cdot)]\quad\text{and}\quad G_\beta(\omega,1) = - G_\beta(T\omega,-1) = \beta V(T\omega) - \beta\mathbb{E}[V(\cdot)].
\end{equation}
Then, $\mathbb{E}[G_\beta(\cdot,\pm 1)] = 0$, and $G_\beta$ is a bounded and centered cocycle (see Definition \ref{cencocdef} in Appendix \ref{app_cocycle}).
Analogous to $F_{\beta,\theta}$ (see Proposition \ref{cikmazde}) in the case of no control, $G_\beta$ will serve as the corrector in the case of full control.

For every $p\in[0,1]$, let
$$g_{\beta,\theta}(\omega,p) = pe^{\beta V(\omega) + \theta + G_\beta(\omega,1)} + (1-p)e^{\beta V(\omega) - \theta + G_\beta(\omega,-1)}.$$

\begin{lemma}\label{cokgenold}
If $\theta\ge\beta\mathbb{E}[V(\cdot)]$, then
\begin{equation}\label{fullineq}
g_{\beta,\theta}(\omega,p) \ge g_{\beta,\theta}(\omega,0) = e^{\beta\mathbb{E}[V(\cdot)] - \theta}
\end{equation}
for every $p\in[0,1]$ and $\omega\in\Omega$.
\end{lemma}

\begin{proof}
Since $V(\cdot)\ge0$, we have
$$2\theta \ge 2\beta\mathbb{E}[V(\cdot)] \ge -\beta V(\omega) + \beta\mathbb{E}[V(\cdot)] -\beta V(T\omega) + \beta\mathbb{E}[V(\cdot)] = G_\beta(\omega,-1) - G_\beta(\omega,1).$$
Therefore, $\beta V(\omega) + \theta + G_\beta(\omega,1) \ge \beta V(\omega) - \theta + G_\beta(\omega,-1)$, and the inequality in \eqref{fullineq} follows.
The equality in \eqref{fullineq} follows from direct substitution.
\end{proof}

For every $t\ge0$, $x\in\mathbb{R}$ and $\pi\in\mathcal{P}_{[\epin t]}(1)$, 
use Lemma \ref{cencoclem} to 
give the following bound, 
\corOO{where the $o([\epsilon^{-1} t])$ error terms
depend on $(\omega,t,x,\beta)$ but not on $\pi$ :}
\begin{align*}
&E_{[\epin x]}^{\pi,\omega}\left[e^{\beta\sum_{i=0}^{[\epin t]-1} V(T_{X_i}\omega) + \theta X_{[\epin t]}}\right]\nonumber\\
&\quad = E_0^{\pi,\omega}\left[e^{\sum_{i=0}^{[\epin t]-1} [\beta V(T_{[\epin x] + X_i}\omega) + \theta Z_{i+1} + G_\beta(T_{[\epin x] + X_i}\omega,Z_{i+1})]}\right]e^{\theta[\epin x] + o([\epin t])}\nonumber\\
&\quad = E_0^{\pi,\omega}\left[e^{\sum_{i=0}^{[\epin t]-2} [\beta V(T_{[\epin x] + X_i}\omega) + \theta Z_{i+1} + G_\beta(T_{[\epin x] + X_i}\omega,Z_{i+1})]}\right.\nonumber\\
&\hspace{17mm}\left.\times g_{\beta,\theta}\left(T_{[\epin x] + X_{[\epin t]-1}}\omega,\pi_{[\epin t]-1}([\epin t],\omega,[\epin x] + X_{[\epin t]-1},1)\right)\right]e^{\theta[\epin x] + o([\epin t])}\nonumber\\
&\quad \ge E_0^{\pi,\omega}\left[e^{\sum_{i=0}^{[\epin t]-2} [\beta V(T_{[\epin x] + X_i}\omega) + \theta Z_{i+1} + G_\beta(T_{[\epin x] + X_i}\omega,Z_{i+1})]}\right]e^{(\beta\mathbb{E}[V(\cdot)] - \theta) + \theta[\epin x] + o([\epin t])},
\end{align*}
\corOO{and the
  last inequality used Lemma \ref{cokgenold}. Iterating, one obtains} 
\begin{align}
&E_{[\epin x]}^{\pi,\omega}\left[e^{\beta\sum_{i=0}^{[\epin t]-1} V(T_{X_i}\omega) + \theta X_{[\epin t]}}\right]
\ge\cdots\ge e^{[\epin t](\beta\mathbb{E}[V(\cdot)] - \theta) + \theta[\epin x] + o([\epin t])}.\label{takpat}
\end{align}
First taking $\ep\log$ of both sides, then taking infimum over $\pi\in\mathcal{P}_{[\epin t]}(1)$, and finally sending $\ep\to0$, we conclude that
\begin{equation}\label{LBfull1}
\overline H_{1,\beta}^L(\theta,t,x) \ge t(\beta\mathbb{E}[V(\cdot)] - \theta) + \theta x.
\end{equation}

\subsubsection{Lower bound when $0<\theta<\beta\mathbb{E}[V(\cdot)]$}
\corOO{We use scaling properties.}
Let $\bar\beta = \bar\beta(\theta) = \frac{\theta}{\mathbb{E}[V(\cdot)]} < \beta$. Then,
\begin{equation}\label{anlamyap}
\ep\log E_{[\epin x]}^{\pi,\omega}\left[e^{\beta\sum_{i=0}^{[\epin t]-1} V(T_{X_i}\omega) + \theta X_{[\epin t]}}\right] \ge \ep\log E_{[\epin x]}^{\pi,\omega}\left[e^{\bar\beta\sum_{i=0}^{[\epin t]-1} V(T_{X_i}\omega) + \theta X_{[\epin t]}}\right]
\end{equation}
and
\begin{equation}\label{LBfull2}
\overline H_{1,\beta}^L(\theta,t,x) \ge \overline H_{1,\bar\beta}^L(\theta,t,x) \ge t(\bar\beta\mathbb{E}[V(\cdot)] - \theta) + \theta x = \theta x
\end{equation}
for every $t\ge0$ and $x\in\mathbb{R}$. Here, the first inequality uses the fact that $V(\cdot)\ge0$, and the second inequality follows from \eqref{LBfull1} which is applicable since $\theta = \bar\beta\mathbb{E}[V(\cdot)]$.

\subsubsection{Lower bound when $\theta = 0$}

Since $V(\cdot)\ge0$, we have
\begin{equation}\label{anlamboz}
\ep\log E_{[\epin x]}^{\pi,\omega}\left[e^{\beta\sum_{i=0}^{[\epin t]-1} V(T_{X_i}\omega) + \theta X_{[\epin t]}}\right] = \ep\log E_{[\epin x]}^{\pi,\omega}\left[e^{\beta\sum_{i=0}^{[\epin t]-1} V(T_{X_i}\omega)}\right] \ge 0
\end{equation}
for every $\ep>0$, $t\ge0$ and $x\in\mathbb{R}$. Taking $\ep\to0$, we conclude that
\begin{equation}\label{LBfull3}
\overline H_{1,\beta}^L(0,t,x)\ge 0.
\end{equation}

\subsection{The effective Hamiltonian}

\begin{proof}[Proof of Theorem \ref{thmfull}]
If $0\le\theta<\beta\mathbb{E}[V(\cdot)]$, then the bounds \eqref{UBfull1}, \eqref{LBfull2} and \eqref{LBfull3} match for every $t>0$ and $x\in\mathbb{R}$,
\begin{equation}\label{omkoc1}
\overline H_{1,\beta}^L(\theta,t,x) = \overline H_{1,\beta}^U(\theta,t,x) = \theta x,
\end{equation}
and taking the infimum in \eqref{beforelimit} over the set $\{\pitwo:\,0<h<h_0\}$ for any $h_0>0$ does not change the limit in \eqref{homlimit}. (Regarding the choice of $x_*$, see Remark \ref{dundundur}.)

If $\theta\ge\beta\mathbb{E}[V(\cdot)]$, then the bounds \eqref{UBfull2} and \eqref{LBfull1} match for every $t>0$ and $x\in\mathbb{R}$,
\begin{equation}\label{omkoc2}
\overline H_{1,\beta}^L(\theta,t,x) = \overline H_{1,\beta}^U(\theta,t,x) = t(\beta\mathbb{E}[V(\cdot)] - \theta) + \theta x,
\end{equation}
and $\pileft$ is asymptotically optimal as $\ep\to0$.

If $\theta<0$, the analogous results follow from symmetry. The existence of the limit in \eqref{homlimit} and the validity of the identity in \eqref{simdigel} follow immediately from \eqref{omkoc1} and \eqref{omkoc2}. Finally, setting $t=1$ and $x=0$, we deduce \eqref{fulllimit}.
\end{proof}

\section{Partial control: Alternative formulation and upper bounds}

In this section, we consider the case $\delta\in(0,1)$ under the assumptions \eqref{ass_wlog} and \eqref{ass_pan}.

\subsection{Alternative formulation}\label{hssira}

Recall from our discussion in Section \ref{homressec},
\corOO{cf.\ \eqref{kisamioldu}},
that the infimum in \eqref{beforelimit} can be taken over
$$\mathcal{P}_n^{BB}(\delta) = \{\pi\in\mathcal{P}_n(\delta):\,{\textstyle \pi_i(n,\omega,y,1) = \frac{1 \pm \delta}{2}}\ \text{for every $i\in[0,n-1]$, $\omega\in\Omega$ and $y\in\mathbb{Z}$}\},$$
i.e., the set of bang-bang policies. 
For every $\pi\in\mathcal{P}_n^{BB}(\delta)$, define $\alpha = (\alpha_0,\ldots,\alpha_{n-1})$ by setting
$$\alpha_i = \alpha_i(n,\omega,y) = \frac1{2}\log\left(\frac{\pi_i(n,\omega,y,1)}{\pi_i(n,\omega,y,-1)}\right) = \begin{cases}
\ \; \, c&\ \text{if}\ \pi_i(n,\omega,y,1) = \frac{1 + \delta}{2},\\
-c&\ \text{if}\ \pi_i(n,\omega,y,1) = \frac{1 - \delta}{2}.
\end{cases}$$
The parameter $c$ was introduced in \eqref{parasi}. Note that 
$$\pi_i(n,\omega,y,\pm1) = \frac{e^{\pm\alpha_i(n,\omega,y)}}{e^{c} + e^{-c}}.$$
\corOO{We}
perform a change of measure: for every $x\in\mathbb{Z}$,
$$E_x^{\pi,\omega}\left[e^{\beta\sum_{i=0}^{n-1} V(T_{X_i}\omega) + \theta X_n}\right] = E_x\left[e^{\sum_{i=0}^{n-1} [\beta V(T_{X_i}\omega) + (\theta + \alpha_i(n,\omega,X_i))Z_{i+1}]}\right]e^{\theta x - n\log\cosh(c)}.$$
Then, \eqref{beforelimit} becomes
\begin{equation}\label{dualform}
u(n,x,\omega\,|\,\delta,\beta,\theta) = \inf_{\alpha\in\mathcal{A}_n(c)}\log E_x\left[e^{\sum_{i=0}^{n-1} [\beta V(T_{X_i}\omega) + (\theta + \alpha_i(n,\omega,X_i))Z_{i+1}]}\right] + \theta x - n\log\cosh(c),
\end{equation}
where the infimum is taken over
\begin{equation}\label{eyensi}
\mathcal{A}_n(c) = \left\{\alpha = (\alpha_0,\ldots,\alpha_{n-1}):\alpha_i = \alpha_i(n,\omega,y) = \pm c\ \text{for every $i\in[0,n-1]$, $\omega\in\Omega$ and $y\in\mathbb{Z}$}\right\}.
\end{equation}
Similarly, \eqref{homlimsup} and \eqref{homliminf} become
\begin{align}
&\begin{aligned}\label{dualformsup}
\overline H_{\delta,\beta}^U(\theta,t,x) &= \limsup_{\ep\to0}\inf_{\alpha\in\mathcal{A}_{[\epin t]}(c)}\ep\log E_{[\epin x]}\left[e^{\sum_{i=0}^{[\epin t]-1} [\beta V(T_{X_i}\omega) + (\theta + \alpha_i([\epin t],\omega,X_i))Z_{i+1}]}\right]\\
&\quad + \theta x - t\log\cosh(c)\quad\text{and}
\end{aligned}\\
&\begin{aligned}\label{dualforminf}
\overline H_{\delta,\beta}^L(\theta,t,x) &= \liminf_{\ep\to0}\inf_{\alpha\in\mathcal{A}_{[\epin t]}(c)}\ep\log E_{[\epin x]}\left[e^{\sum_{i=0}^{[\epin t]-1} [\beta V(T_{X_i}\omega) + (\theta + \alpha_i([\epin t],\omega,X_i))Z_{i+1}]}\right]\\
&\quad + \theta x - t\log\cosh(c).
\end{aligned}
\end{align}

\subsection{General upper bound}

The policies $\pileft,\piright\in\mathcal{P}_n^{BB}(\delta)$ (given in \eqref{sagsol}) correspond to $\aleft,\aright\in\mathcal{A}_n(c)$ with
$$\aleft_i(n,\omega,y) \equiv -c\quad\text{and}\quad\aright_i(n,\omega,y) \equiv c,$$
respectively. For every $\theta\in\mathbb{R}$, $t>0$ and $x\in\mathbb{R}$, substituting each of these policies (with $n=[\epin t]$) in the expectation on the right-hand side of \eqref{dualformsup} and using Theorem \ref{thmno}, we deduce the following bound:
\begin{equation}\label{UBpart1}
\overline H_{\delta,\beta}^U(\theta,t,x) \le t\left(\min\{\Lambda_\beta(\theta - c), \Lambda_\beta(\theta + c)\} - \log\cosh(c)\right) + \theta x.
\end{equation}

\subsection{Upper bound when $|\theta| \le c$}\label{cinsisimhay}

For every $h\in(0,1)$, $\ell\in\mathbb{N}$, $x\in\mathbb{R}$ and $a>0$, the RW policy $\pieL\in\mathcal{P}_n^{BB}(\delta)$ (given in \eqref{pengol}) 
corresponds to $\aeL\in\mathcal{A}_n(c)$ with
$$\aeL_i(n,\omega,y) = \begin{cases}\ \; \,c&\ \text{if $y < x_*$},\\
-c&\ \text{if $y \ge x_*$}.\end{cases}$$
When the walk starts at $[\epin x]$ with a sufficiently small $\ep>0$, recall from Lemma \ref{gayettehos} and Remark \ref{dundundur} that $[x_*-\ell,x_*+\ell-1]\subset[\epin(x-a),\epin(x+a)]$. 
Assume without loss of generality that $x_* = [\epin x]$, i.e., $[[\epin x]-\ell,[\epin x] + \ell-1]$ is an $h$-valley. (Starting the walk from $x_*$ instead of $[\epin x]$ changes the right-hand side of \eqref{homlimsup} by at most $a(\log 2 + \beta + 2|\theta| + 2c)$, which goes to $0$ as $a\to0$.) When $\theta = c$, substituting $\aeL$ in the expectation on the right-hand side of \eqref{dualformsup}, we get
\begin{equation}\label{gazagel}
	\overline H_{\delta,\beta}^U(c,t,x) \le \limsup_{\ep\to0}\ep\log E_0\left[e^{\sum_{i=0}^{[\epin t]-1} [\beta V(T_{[\epin x] + X_i}\omega) + 2cZ_{i+1}\one_{\{X_i < 0\}}]}\right] + cx - t\log\cosh(c),
\end{equation}
where we shifted the starting point of the RW $(X_i)_{i\ge0}$ to the origin.

Due to each complete left excursion starting from the origin, the $\sum 2cZ_{i+1}\one_{\{X_i < 0\}}$ term in the exponent inside the expectation on the right-hand side of \eqref{gazagel} increases precisely by $2c$, and this sum does not increase (but it can decrease) due to an incomplete left excursion. On the other hand, complete and incomplete right excursions starting from the origin have no effect on this sum. We deduce that
\begin{equation}\label{protoUB}
E_0\left[e^{\sum_{i=0}^{[\epin t]-1}[\beta V(T_{[\epin x] + X_i}\omega) + 2cZ_{i+1}\one_{\{X_i < 0\}}]}\right] \le E_0\left[e^{\beta\sum_{i=0}^{[\epin t]-1}V(T_{[\epin x] + X_i}\omega) + 2c\mathcal{L}_0(X_{0,[\epin t]})}\right],
\end{equation}
where 
\begin{equation}\label{elemtere}
\mathcal{L}_0(x_{0,n}) = \sum_{i=1}^n\one_{\{x_{i-1} = -1, x_i = 0\}}
\end{equation}
counts the number of complete left excursions of a nearest-neighbor path $x_{0,n}$ with $x_0 = 0$ and $n\in\mathbb{N}$.

For every $j,k\in\mathbb{N}\cup\{0\}$, let $S_j = \sum_{i=1}^j\one_{\{-\ell \le X_i \le\ell-1\}}$, $\sigma_k = \inf\{i\ge0: S_i = k\}$ and $Y_k^\ell = X_{\sigma_k}$. It is easy to see that $(Y_k^\ell)_{k\ge0}$ is a Markov process on $[-\ell,\ell-1]$ starting from the origin, and it has the following transition probabilities:
\begin{equation}
\begin{aligned}\label{tantana}
P_0(Y_k^\ell = y - 1\,|\,Y_{k-1}^\ell = y) &= P_0(Y_k^\ell = y + 1\,|\,Y_{k-1}^\ell = y) = 1/2\quad\text{if $y\in[-\ell+1,\ell-2]$,}\\
P_0(Y_k^\ell = -\ell\,|\,Y_{k-1}^\ell = -\ell) &= P_0(Y_k^\ell = -\ell + 1\,|\,Y_{k-1}^\ell = -\ell) = 1/2,\quad\text{and}\\
P_0(Y_k^\ell = \ell-1\,|\,Y_{k-1}^\ell = \ell-1) &= P_0(Y_k^\ell = \ell-2\,|\,Y_{k-1}^\ell = \ell-1) = 1/2.
\end{aligned}
\end{equation}
In words, $(Y_k^\ell)_{k\ge0}$ is a reflected RW on $[-\ell,\ell-1]$ and subject to geometric holding times (with rate $1/2$) at $-\ell$ and $\ell-1$. 
With this notation and observations, we control the right-hand side of \eqref{protoUB} as follows:
\begin{align}
&E_0\left[e^{\beta\sum_{i=0}^{[\epin t]-1}V(T_{[\epin x] + X_i}\omega) + 2c\mathcal{L}_0(X_{0,[\epin t]})}\right]\nonumber\\
&\quad = \sum_{m=1}^{[\epin t]}E_0\left[e^{\beta\sum_{i=0}^{[\epin t]-1}V(T_{[\epin x] + X_i}\omega) + 2c\mathcal{L}_0(X_{0,[\epin t]})}\one_{\{S_{[\epin t]} = m\}}\right]\nonumber\\
&\quad\le \sum_{m=1}^{[\epin t]}E_0\left[e^{2c\mathcal{L}_0(X_{0,[\epin t]})}\one_{\{S_{[\epin t]} = m\}}\right]e^{mh\beta + ([\epin t]-m)\beta}\nonumber\\
&\quad = \sum_{m=1}^{[\epin t]}E_0\left[e^{2c\mathcal{L}_0(Y_{0,m}^\ell)}\one_{\{S_{[\epin t]} = m\}}\right]e^{mh\beta + ([\epin t]-m)\beta}\nonumber\\
&\quad\le \sum_{m=1}^{[\epin t]} E_0\left[e^{2c\mathcal{L}_0(Y_{0,m}^\ell)}\right]e^{mh\beta + ([\epin t]-m)\beta} \le \sum_{m=1}^{[\epin t]} e^{mJ_\ell(2c) + mh\beta + ([\epin t]-m)\beta + o(m)}\nonumber\\
&\quad = e^{[\epin t]\max\left\{\beta,J_\ell(2c) + h\beta\right\} + o([\epin t])}.\label{konsek}
\end{align}
Here,
$$J_\ell(2c) := \limsup_{m\to\infty}\frac1{m}\log E_0\left[e^{2c\mathcal{L}_0(Y_{0,m}^\ell)}\right].$$
The proof of the following proposition is deferred to Appendix \ref{app_simple}.

\begin{proposition}\label{countexcursion}
For every $c\in(0,\infty)$, the limit
$$J(2c) := \lim_{n\to\infty}\frac1{n}\log E_0\left[e^{2c\mathcal{L}_0(X_{0,n})}\right]$$
exists. Moreover,
$$\lim_{\ell\to\infty}J_\ell(2c) = J(2c) = \log\cosh(c).$$
\end{proposition}

Putting together \eqref{gazagel}, \eqref{protoUB} and \eqref{konsek} (and adapting the same argument to the $\theta = -c$ case), we get
\begin{equation}
\begin{aligned}\label{UBcece}
\overline H_{\delta,\beta}^U(\pm c,t,x) &\le \liminf_{h\to 0}\liminf_{\ell\to\infty}\limsup_{\ep\to0}\ep\log E_{[\epin x]}^{\pieL,\omega}\left[e^{\beta\sum_{i=0}^{[\epin t]-1}V(T_{X_i}\omega) \pm c X_{[\epin t]}}\right]\\ &\le \lim_{h\to 0}\lim_{\ell\to\infty}t\left(\max\left\{\beta,J_\ell(2c) + h\beta\right\} - \log\cosh(c)\right) \pm cx\\
&= t[\beta - \log\cosh(c)]^+ \pm cx.
\end{aligned}
\end{equation}

Finally, for any $\theta\in [-c,c]$, let $r = \frac{\theta +c}{2c}\in [0,1]$. Then, $\theta = (1-r)(-c) + rc$ is a convex combination. For every $h\in(0,1)$ and $\ell\in\mathbb{N}$, H\"older's inequality gives
\begin{align*}
&E_{[\epin x]}^{\pieL,\omega}\left[e^{\beta\sum_{i=0}^{[\epin t]-1} V(T_{X_i}\omega) + \theta X_{[\epin t]}}\right]\\ &\ \le \left(E_{[\epin x]}^{\pieL,\omega}\left[e^{\beta\sum_{i=0}^{[\epin t]-1} V(T_{X_i}\omega) - c X_{[\epin t]}}\right]\right)^{(1-r)} \left(E_{[\epin x]}^{\pieL,\omega}\left[e^{\beta\sum_{i=0}^{[\epin t]-1} V(T_{X_i}\omega) + c X_{[\epin t]}}\right]\right)^r.
\end{align*}
Therefore, by \eqref{UBcece},
\begin{equation}
\begin{aligned}\label{UBpart2}
\overline H_{\delta,\beta}^U(\theta,t,x) &\le \liminf_{h\to 0}\liminf_{\ell\to\infty}\limsup_{\ep\to0}\ep\log E_{[\epin x]}^{\pieL,\omega}\left[e^{\beta\sum_{i=0}^{[\epin t]-1}V(T_{X_i}\omega) + \theta X_{[\epin t]}}\right]\\
&\le t[\beta - \log\cosh(c)]^+ + \theta x.
\end{aligned}
\end{equation}

\section{Partial control: Lower bounds and the effective Hamiltonian}

As in the previous section, we consider the case $\delta\in(0,1)$ under the assumptions \eqref{ass_wlog} and \eqref{ass_pan}.

\subsection{Uniform lower bound}\label{martingale}

For every $\ep>0$, $t>0$, $x\in\mathbb{R}$, $\omega\in\Omega$, $\theta\in\mathbb{R}$ and $\alpha\in\mathcal{A}_{[\epin t]}(c)$ (see \eqref{eyensi}),
$$M_j = M_j(\omega,\theta,\alpha) := \sum_{i=0}^{j-1}(\theta + \alpha_i([\epin t],\omega,X_i))Z_{i+1}$$ defines a martingale $(M_j)_{0\le j\le[\epin t]}$ under $P_{[\epin x]}$, with $|M_j - M_{j-1}| \le |\theta| + c$. Therefore, for every $b>0$,
$$P_{[\epin x]}(M_{[\epin t]} \le - [\epin t]b) \le e^{-\frac{[\epin t]b^2}{2(|\theta| + c)^2}}$$
by the Azuma-Hoeffding inequality (see \cite[Section 12.2]{GriSti2001}).

For every $h\in(0,1)$, $\ell\in\mathbb{N}$, $a\in(0,t)$, $\mathbb{P}$-a.e.\ $\omega$ and sufficiently small $\ep>0$, we know by Lemma \ref{gayettehos} that there exists an $h$-hill of the form $[x^* - \ell,x^* + \ell - 1]$ contained in $[\epin(x-a),\epin(x+a)]$. Recall from the proof of Lemma \ref{kendiyag} that $A_{[\epin t]}(x^*,h,\ell)$
is the event that the particle marches deterministically from 
\corOO{$[\epin x]$}
to $x^*$ and then spends the rest of the $[\epin t]$ units of time in this $h$-hill. It follows from Lemma \ref{appasakla} that
$$P_{[\epin x]}(A_{[\epin t]}(x^*,h,\ell)) \ge e^{[\epin t]\log\cos(\pi/(2\ell + 1)) - \epin a\log2 + o(\epin t)}.$$
Take $a\in(0,t)$ sufficiently small and $\ell\in\mathbb{N}$ sufficiently large (both depending on $b$) so that $$t|\log\cos(\pi/(2\ell + 1))| + a\log2 < \frac{tb^2}{2(|\theta| + c)^2}.$$
Then, $$P_{[\epin x]}(A_{[\epin t]}(x^*,h,\ell)\setminus\{M_{[\epin t]}\le - [\epin t]b\}) \ge e^{[\epin t]\log\cos(\pi/(2\ell + 1)) - \epin a\log2 + o(\epin t)}.$$
Restricting the expectation on the right-hand side of \eqref{dualforminf} on this set difference gives
\begin{align*}
&E_{[\epin x]}\left[e^{\sum_{i=0}^{[\epin t]-1} [\beta V(T_{X_i}\omega) + (\theta + \alpha_i([\epin t],\omega,X_i))Z_{i+1}]}\right]\\
&\quad\ge E_{[\epin x]}\left[e^{\beta\sum_{i=0}^{[\epin t]-1}V(T_{X_i}\omega) + M_{[\epin t]}}\one_{A_{[\epin t]}(x^*,h,\ell)\setminus\{M_{[\epin t]}\le - [\epin t]b\}}\right]\\
&\quad\ge e^{[\epin t]\left(\beta h + \log\cos(\pi/(2\ell + 1)) - b\right) - \epin a(\beta h + \log2) + o(\epin t)}.
\end{align*}
Taking $\ep\log$ of both sides, then sending $\ep\to0$, $a\to0$, $h\to 1$, $\ell\to\infty$, and finally taking $b\to0$, we get the following uniform lower bound:
\begin{equation}\label{LBpart1}
\overline H_{\delta,\beta}^L(\theta,t,x) \ge t(\beta - \log\cosh(c)) + \theta x.
\end{equation}


%
%

\subsection{Lower bounds when $\theta\ge0$ and $\Lambda_\beta(\theta - c)>\beta$}\label{muspet}

\corOO{We begin with a preliminary computation.}
For every $\xi\in\mathbb{R}$, let
$$g_{\beta,\theta-c}(\omega,\xi) = \frac1{2}e^{\beta V(\omega) + \xi + F_{\beta,\theta-c}(\omega,1)} + \frac1{2}e^{\beta V(\omega) - \xi + F_{\beta,\theta-c}(\omega,-1)},$$
where $F_{\beta,\theta-c}$ is the corrector defined in \eqref{gaztap}.

\begin{lemma}\label{derin}
If $\Lambda_\beta(\theta - c) > \beta$, then
\begin{equation}\label{yemekhane}
g_{\beta,\theta-c}(\omega,\theta-c) = e^{\Lambda_\beta(\theta-c)}
\end{equation}
for every $\omega\in\Omega$. Moreover, the following equivalence holds:
\begin{equation}\label{eskinut}
g_{\beta,\theta-c}(\omega,\theta+c) \ge g_{\beta,\theta-c}(\omega,\theta-c)\quad\iff\quad\theta + F_{\beta,\theta-c}(\omega,1) \ge -\theta + F_{\beta,\theta-c}(\omega,-1).
\end{equation}
\end{lemma}

\begin{proof}
The equality in \eqref{yemekhane} is immediate from \eqref{menzilde}. The equivalence in \eqref{eskinut} is shown as follows:
\begin{align}
&g_{\beta,\theta-c}(\omega,\theta+c) \ge g_{\beta,\theta-c}(\omega,\theta-c)\nonumber\\
\iff& \frac1{2}e^{(\theta +c) + F_{\beta,\theta-c}(\omega,1)} + \frac1{2}e^{- (\theta +c) + F_{\beta,\theta-c}(\omega,-1)} \ge \frac1{2}e^{(\theta -c) + F_{\beta,\theta-c}(\omega,1)} + \frac1{2}e^{- (\theta -c) + F_{\beta,\theta-c}(\omega,-1)}\nonumber\\
\iff& e^{\theta + F_{\beta,\theta-c}(\omega,1)}\left(\frac{e^c - e^{-c}}{2}\right) \ge e^{-\theta + F_{\beta,\theta-c}(\omega,-1)}\left(\frac{e^c - e^{-c}}{2}\right)\nonumber\\
\iff& \theta + F_{\beta,\theta-c}(\omega,1) \ge -\theta + F_{\beta,\theta-c}(\omega,-1)\nonumber.\qedhere
\end{align}
\end{proof}

\subsubsection{Lower bound when $\theta > c$ and $\Lambda_\beta(\theta - c) > \beta$}

For every $\omega\in\Omega$,
$$\theta + F_{\beta,\theta-c}(\omega,1) > c > 0 > -c > - (\theta + F_{\beta,\theta-c}(T_{-1}\omega,1)) = -\theta + F_{\beta,\theta-c}(\omega,-1)$$
holds by \eqref{fboundsrec}. Hence,
\begin{equation}\label{budaynim}
g_{\beta,\theta-c}(\omega,\theta+c) \ge g_{\beta,\theta-c}(\omega,\theta-c) = e^{\Lambda_\beta(\theta-c)}
\end{equation}
by Lemma \ref{derin}. Therefore, for every $\ep>0$, $t>0$, $x\in\mathbb{R}$, $\alpha\in\mathcal{A}_{[\epin t]}(c)$ and $\mathbb{P}$-a.e.\ $\omega$,
\begin{align}
& E_{[\epin x]}\left[e^{\sum_{i=0}^{[\epin t]-1} [\beta V(T_{X_i}\omega) + (\theta + \alpha_i([\epin t],\omega,X_i))Z_{i+1}]}\right]\nonumber\\
&\quad = E_0\left[e^{\sum_{i=0}^{[\epin t]-1} [\beta V(T_{[\epin x] + X_i}\omega) + (\theta + \alpha_i([\epin t],\omega,[\epin x] + X_i))Z_{i+1} + F_{\beta,\theta-c}(T_{[\epin x] + X_i}\omega,Z_{i+1})]}\right]e^{o(\epin t)}\nonumber\\
&\quad = E_0\left[e^{\sum_{i=0}^{[\epin t]-2} [\beta V(T_{[\epin x] + X_i}\omega) + (\theta + \alpha_i([\epin t],\omega,[\epin x] + X_i))Z_{i+1} + F_{\beta,\theta-c}(T_{[\epin x] + X_i}\omega,Z_{i+1})]}\right.\nonumber\\
&\hspace{14mm}\left.\times g_{\beta,\theta-c}(T_{[\epin x] + X_{[\epin t]-1}}\omega,\theta + \alpha_{[\epin t]-1}([\epin t],\omega,[\epin x] + X_{[\epin t]-1}))\right]e^{o(\epin t)}\nonumber\\
&\quad \ge E_0\left[e^{\sum_{i=0}^{[\epin t]-2} [\beta V(T_{[\epin x] + X_i}\omega) + (\theta + \alpha_i([\epin t],\omega,[\epin x] + X_i))Z_{i+1} + F_{\beta,\theta-c}(T_{[\epin x] + X_i}\omega,Z_{i+1})]}\right]e^{\Lambda_\beta(\theta - c) + o(\epin t)}\nonumber\\
&\quad \ge \cdots \ge e^{[\epin t]\Lambda_\beta(\theta - c) + o(\epin t)}.\nonumber
\end{align}
Here, the first equality follows from Lemma \ref{cencoclem} (in Appendix \ref{app_cocycle}). Recalling \eqref{dualforminf}, we conclude that
\begin{equation}\label{LBpart2}
\overline H_{\delta,\beta}^L(\theta,t,x) \ge t(\Lambda_\beta(\theta - c) - \log\cosh(c)) + \theta x.
\end{equation}

\subsubsection{Lower bound when $0 < \theta < c$ and $\beta < \Lambda_\beta(\theta - c) \le \log\cosh(c)$}\label{bundanibaret}

For every $\omega\in\Omega$,
\begin{align*}
(\theta -c) - F_{\beta,\theta-c}(\omega,-1) \ge \log E_0\left[e^{-\Lambda_\beta(\theta - c)\tau_1}\one_{\{\tau_1<\infty\}}\right] &\ge \log E_0\left[e^{-\log\cosh(c)\tau_1}\one_{\{\tau_1<\infty\}}\right]\\
&= \log\left(\cosh(c) - \sqrt{\cosh^2(c) - 1}\right) = -c.
\end{align*}
Here, the first inequality follows from \eqref{fboundsrec}, and the first equality is shown in \eqref{duzel1} (in Appendix \ref{app_simple}). Therefore,
$$\theta - F_{\beta,\theta-c}(\omega,-1) \ge 0 \ge -(\theta - F_{\beta,\theta-c}(T_1\omega,-1)) = -\theta - F_{\beta,\theta-c}(\omega,1),$$
and \eqref{budaynim} follows from Lemma \ref{derin}. Hence, the argument immediately below \eqref{budaynim} is applicable, and
$$E_{[\epin x]}\left[e^{\sum_{i=0}^{[\epin t]-1} [\beta V(T_{X_i}\omega) + (\theta + \alpha_i([\epin t],\omega,X_i))Z_{i+1}]}\right] \ge e^{[\epin t]\Lambda_\beta(\theta - c) + o(\epin t)}$$
for every $\ep>0$, $t>0$, $x\in\mathbb{R}$, $\alpha\in\mathcal{A}_{[\epin t]}(c)$ and $\mathbb{P}$-a.e.\ $\omega$. Recalling \eqref{dualforminf} as before, we conclude that
\begin{equation}\label{LBpart3}
\overline H_{\delta,\beta}^L(\theta,t,x) \ge t(\Lambda_\beta(\theta - c) - \log\cosh(c)) + \theta x.
\end{equation}

\subsubsection{Lower bound when $0 < \theta < c$ and $\beta < \log\cosh(c) < \Lambda_\beta(\theta - c)$}\label{bidebuvar}

It follows from Proposition \ref{temelsaf}(a,b,c) and the intermediate value theorem that there exists a unique $\bar\theta(\beta,c)\in(0,c)$ such that
$$\Lambda_\beta(\bar\theta(\beta,c)-c) = \log\cosh(c).$$
By Proposition \ref{temelsaf}(a), the map $\beta\mapsto\bar\theta(\beta,c)$ is increasing for $\beta\in(0,\log\cosh(c))$, with $\bar{\theta}(0^+,c) = 0$. For every $\theta\in(0,\bar\theta(\beta,c))$, there is a unique $\bar\beta = \bar\beta(\theta,c) \in(0,\beta)$ such that $\theta = \bar\theta(\bar\beta,c)$. Using these quantities, we get the following bound: for every $t>0$ and $x\in\mathbb{R}$,
\begin{equation}
\begin{aligned}\label{LBpart4}
\overline H_{\delta,\beta}^L(\theta,t,x) \ge \overline H_{\delta,\bar\beta}^L(\theta,t,x) &\ge t(\Lambda_{\bar\beta}(\theta - c) - \log\cosh(c)) + \theta x\\
&= t(\Lambda_{\bar\beta}(\bar\theta(\bar\beta,c) - c) - \log\cosh(c)) + \theta x = \theta x.
\end{aligned}
\end{equation}
Here, the first inequality uses the fact that $V(\cdot)\ge0$, and the second inequality follows from \eqref{LBpart3} which is applicable since $\bar\beta < \Lambda_{\bar\beta}(\theta - c) = \log\cosh(c)$.

\subsubsection{Lower bound when $\theta = 0$}

Since $V(\cdot)\ge0$, it is clear from \eqref{homliminf} that, for every $t>0$ and $x\in\mathbb{R}$,
\begin{equation}\label{LBpart5}
\overline H_{\delta,\beta}^L(0,t,x)\ge 0.
\end{equation}

\subsection{The effective Hamiltonian}

\begin{proof}[Proof of Theorem \ref{thmweak}]
If $0\le\theta\le c$, then the bounds \eqref{UBpart2} and \eqref{LBpart1} match for every $t>0$ and $x\in\mathbb{R}$,
\begin{equation}\label{rahkoc1}
\overline H_{\delta,\beta}^L(\theta,t,x) = \overline H_{\delta,\beta}^U(\theta,t,x) = t(\beta - \log\cosh(c)) + \theta x,
\end{equation}
and taking the infimum in \eqref{beforelimit} over the set $\{\pieL:\,0<h<h_0,\ \ell>\ell_0\}$ for any $h_0>0$ and $\ell_0\in\mathbb{N}$ does not change the limit in \eqref{homlimit}. (Regarding the choice of $x_*$, see Remark \ref{dundundur}.)

If $\theta \ge 0$ and $\Lambda_\beta(\theta-c) = \beta$, then the bounds \eqref{UBpart1} and \eqref{LBpart1} match for every $t>0$ and $x\in\mathbb{R}$,
\begin{equation}\label{rahkoc2} 
\overline H_{\delta,\beta}^L(\theta,t,x) = \overline H_{\delta,\beta}^U(\theta,t,x) = t(\beta - \log\cosh(c)) + \theta x,
\end{equation}
and $\pileft$ is asymptotically optimal as $\ep\to0$.

If $\theta > c$ and $\Lambda_\beta(\theta-c) > \beta$, then the bounds \eqref{UBpart1} and \eqref{LBpart2} match for every $t>0$ and $x\in\mathbb{R}$,
\begin{equation}\label{rahkoc3}
\overline H_{\delta,\beta}^L(\theta,t,x) = \overline H_{\delta,\beta}^U(\theta,t,x) = t(\Lambda_\beta(\theta-c) - \log\cosh(c)) + \theta x,
\end{equation}
and $\pileft$ is asymptotically optimal as $\ep\to0$.

If $\theta<0$, the analogous results follow from symmetry. The existence of the limit in \eqref{homlimit} and the validity of the identity in \eqref{simdigel} follow immediately from \eqref{rahkoc1}, \eqref{rahkoc2} and \eqref{rahkoc3}. Finally, setting $t=1$ and $x=0$, we deduce \eqref{weaklimit}.
\end{proof}

\begin{proof}[Proof of Theorem \ref{thmstrong}]
	
Recall from Section \ref{bidebuvar} that there exists a unique $\bar\theta(\beta,c)\in(0,c)$ such that
$$\Lambda_\beta(\bar\theta(\beta,c)-c) = \log\cosh(c).$$

If $0\le\theta < \bar\theta(\beta,c)$, then the bounds \eqref{UBpart2}, \eqref{LBpart4} and \eqref{LBpart5} match for every $t>0$ and $x\in\mathbb{R}$,
\begin{equation}\label{sunkrc1} 
\overline H_{\delta,\beta}^L(\theta,t,x) = \overline H_{\delta,\beta}^U(\theta,t,x) = \theta x,
\end{equation}
and taking the infimum in \eqref{beforelimit} over the set $\{\pieL:\,0<h<h_0,\ \ell>\ell_0\}$ for any $h_0>0$ and $\ell_0\in\mathbb{N}$ does not change the limit in \eqref{homlimit}. (Regarding the choice of $x_*$, see Remark \ref{dundundur}.)

If $\theta \ge 0$ and $\Lambda_\beta(\theta-c) = \beta$, then the bounds \eqref{UBpart1} and \eqref{LBpart1} match for every $t>0$ and $x\in\mathbb{R}$,
\begin{equation}\label{sunkrc2} 
\overline H_{\delta,\beta}^L(\theta,t,x) = \overline H_{\delta,\beta}^U(\theta,t,x) = t(\beta - \log\cosh(c)) + \theta x,
\end{equation}
and $\pileft$ is asymptotically optimal as $\ep\to0$.

If $\theta \ge \bar\theta(\beta,c)$ and $\Lambda_\beta(\theta-c) > \beta$, then the bounds \eqref{UBpart1}, \eqref{LBpart2} and \eqref{LBpart3} match for every $t>0$ and $x\in\mathbb{R}$,
\begin{equation}\label{sunkrc3} 
\overline H_{\delta,\beta}^L(\theta,t,x) = \overline H_{\delta,\beta}^U(\theta,t,x) = t(\Lambda_\beta(\theta-c) - \log\cosh(c)) + \theta x,
\end{equation}
and $\pileft$ is asymptotically optimal as $\ep\to0$.

If $\theta<0$, the analogous results follow from symmetry. The existence of the limit in \eqref{homlimit} and the validity of the identity in \eqref{simdigel} follow immediately from \eqref{sunkrc1}, \eqref{sunkrc2} and \eqref{sunkrc3}. Finally, setting $t=1$ and $x=0$, we deduce \eqref{stronglimit}.
\end{proof}

\section{Homogenization of the Bellman equation}\label{sechomo}

We start with a lemma which excludes the full control regime.

\begin{lemma}\label{cansene}
For every $\omega\in\Omega$, $\delta\in[0,1)$, $\beta > 0$ and $\theta\in\mathbb{R}$, the function $u(\cdot,\cdot,\omega) = u(\cdot,\cdot,\omega\,|\,\delta,\beta,\theta)$ (defined in \eqref{beforelimit}) satisfies the following Lipschitz condition: for every $m,n\in\mathbb{N}$ and $x,y\in\mathbb{Z}$, 
\begin{equation}\label{sakinolzen}
|u(n,x,\omega) - u(m,y,\omega)| \le (\beta + |\theta|)|n-m| + (\beta + |\theta| - \textstyle{\log(\frac{1-\delta}{2})})|x-y|.
\end{equation}
\end{lemma}

\begin{proof}
It follows easily from \eqref{beforelimit} that
\begin{align*}
|u(n,x,\omega) - u(m,x,\omega)| &\le (\beta + |\theta|)|n-m|\quad\text{and}\\
u(m + |x-y|,x,\omega) - u(m,y,\omega) &\ge \textstyle{\log(\frac{1-\delta}{2})|x-y|}
\end{align*}
for every $m,n\in\mathbb{N}$ and $x,y\in\mathbb{Z}$. The second inequality is obtained by considering the event that the particle marches from $x$ to $y$ in $|x-y|$ steps. A suitable combination of these inequalities gives \eqref{sakinolzen}.
\end{proof}

\begin{proof}[Proof of Theorem \ref{unifhom2}]
If $\delta\in[0,1]$, $\beta>0$ and $\theta\in\mathbb{R}$, then for $\mathbb{P}$-a.e.\ $\omega$, $$\lim_{\ep\to0}u_\ep(t,x,\omega) = u_o(t,x) = t\overline H_{\delta,\beta}(\theta) + \theta x$$ for every $t>0$ and $x\in\mathbb{R}$ by Theorems \ref{thmno}, \ref{thmfull}, \ref{thmweak} and \ref{thmstrong}. Moreover, at $t=0$,
$$\lim_{\ep\to0}u_\ep(0,x,\omega) = \lim_{\ep\to0}\ep\theta[\epin x] = \theta x = u_o(0,x).$$
It remains to improve this pointwise limit on $[0,\infty)\times\mathbb{R}$ to a uniform limit on compact subsets of $[0,\infty)\times\mathbb{R}$.

If $\delta\in[0,1)$, then Lemma \ref{cansene} gives the following bounds: for every $\ep>0$, $s,t\ge0$ and $x,y\in\mathbb{R}$,
\begin{equation} 
\begin{aligned}\label{lafstok2}
|u_\ep(t,x,\omega) - u_\ep(s,y,\omega)| 
&\le \ep(\beta + |\theta|)|[\epin t]-[\epin s]| + \ep(\beta + |\theta| - \textstyle{\log(\frac{1-\delta}{2})})|[\epin x] - [\epin y]|\\
&\le (\beta + |\theta|)|t-s| + (\beta + |\theta| - \textstyle{\log(\frac{1-\delta}{2})})|x-y| + \ep(2\beta + 2|\theta| - \textstyle{\log(\frac{1-\delta}{2})}).
\end{aligned}
\end{equation}
For every $\ep'>0$, $t_{max}>0$ and $B>0$, partition the rectangle $[0,t_{max}]\times [-B,B]$ into finitely many (say $N$) identical squares with side length $\frac{\ep'}{12}(\beta + |\theta| - \textstyle{\log(\frac{1-\delta}{2})})^{-1}$. Fix a point $(s_i,y_i)$ in the $i$th square. By pointwise convergence, there exists an $\ep_i>0$ such that $\ep_i(2\beta + 2|\theta| - \textstyle{\log(\frac{1-\delta}{2})}) < \frac{\ep'}{6}$ and $|u_\ep(s_i,y_i,\omega) - u_o(s_i,y_i)| < \ep'/3$ whenever $0<\ep<\ep_i$. If $(t,x)$ is any point in the $i$th square, then
\begin{align*}
|u_\ep(t,x,\omega) - u_o(t,x)| &\le |u_\ep(t,x,\omega) - u_\ep(s_i,y_i,\omega)| + |u_\ep(s_i,y_i,\omega) - u_o(s_i,y_i)| + |u_o(s_i,y_i) - u_o(t,x)|\\
&< \left(\frac{\ep'}{12} + \frac{\ep'}{12} + \frac{\ep'}{6}\right) + \frac{\ep'}{3} + \left(\frac{\ep'}{12} + \frac{\ep'}{12} + \frac{\ep'}{6}\right) = \ep'
\end{align*}
by \eqref{lafstok2}. Taking $\ep_0 = \min\{\ep_1,\ldots,\ep_N\}$ concludes the proof of uniform convergence on $[0,t_{max}]\times[-B,B]$.

If $\delta = 1$, then the walk under bang-bang policies is not elliptic and Lemma \ref{cansene} is not applicable. Therefore, we prove the desired uniform convergence by revisiting Sections \ref{apirband} and \ref{lovirband} where we obtained upper and lower bounds for $u_\ep(t,x,\omega)$ with error bounds. Fix $t_{max}>0$ and $B>0$.
\begin{itemize}
	\item For every $h\in(0,1)$, $a>0$, $\mathbb{P}$-a.e.\ $\omega$ and sufficiently small $\ep>0$ (depending on $\omega,h,a,B$), \eqref{mazharfu} gives
	$$u_\ep(t,x,\omega) \le \theta x + t_{max}\beta h + a(\beta + |\theta|) + \ep|\theta|$$
	uniformly for $(t,x)\in[0,t_{max}]\times[-B,B]$. Here, the uniformity in $x$ comes from Lemma \ref{gayettehos}.
	\item When $\theta\ge 0$, for $\mathbb{P}$-a.e.\ $\omega$, \eqref{fuatoz} gives
    \begin{align*}
    u_\ep(t,x,\omega) &\le t(\beta\mathbb{E}[V(\cdot)] - \theta) + \theta x - \ep\sum_{i=0}^{[\epin t] - 1}G_\beta(T_{[\epin x] - i}\omega,-1) + \ep(\beta + 2|\theta|)\\
    &= t(\beta\mathbb{E}[V(\cdot)] - \theta) + \theta x + \ep o(\epin t_{max}) + \ep(\beta + 2|\theta|)
    \end{align*}
    uniformly for $(t,x)\in[0,t_{max}]\times[-B,B]$. Here, the uniformity in $x$ comes from Lemma \ref{cencoclem} (in Appendix \ref{app_cocycle}) which is applicable since $G_\beta$ (defined in \eqref{ptesi}) is a bounded and centered cocycle.
    \item When $\theta\ge\beta\mathbb{E}[V(\cdot)]$, for $\mathbb{P}$-a.e.\ $\omega$, \eqref{takpat} gives
    \begin{align*}
    u_\ep(t,x,\omega) &\ge \ep[\epin t](\beta\mathbb{E}[V(\cdot)] - \theta) + \ep\theta[\epin x] + \ep o(\epin t_{max})\\
    &\ge t(\beta\mathbb{E}[V(\cdot)] - \theta) + \theta x + \ep o(\epin t_{max}) - \ep(\beta + 2|\theta|)
    \end{align*}
    uniformly for $(t,x)\in[0,t_{max}]\times[-B,B]$. Again, the uniformity in $x$ comes from Lemma \ref{cencoclem}.
    \item When $0<\theta<\beta\mathbb{E}[V(\cdot)]$, \eqref{anlamyap} and the lower bound in the previous case give
    \begin{align*}
    u_\ep(t,x,\omega\,|\,\delta,\beta,\theta) \ge u_\ep(t,x,\omega\,|\,\delta,\bar\beta,\theta) &\ge t(\bar\beta\mathbb{E}[V(\cdot)] - \theta) + \theta x + \ep o(\epin t_{max}) - \ep(\bar\beta + 2|\theta|)\\
    &= \theta x + \ep o(\epin t_{max}) - \ep(\bar\beta + 2|\theta|)
    \end{align*}
    uniformly for $(t,x)\in[0,t_{max}]\times[-B,B]$.
    \item When $\theta = 0$, \eqref{anlamboz} gives
    $$u_\ep(t,x,\omega)\ge 0$$
    uniformly for $(t,x)\in[0,\infty)\times\mathbb{R}$.
\end{itemize}
Combining these upper and lower bounds, uniform convergence on $[0,t_{max}]\times[-B,B]$ follows.
\end{proof}

\section*{Acknowledgments}
We thank E.\ Kosygina for suggesting to us that nonconvex homogenization can be of interest in the discrete setup,
for valuable discussions which motivated this project,
for her help in formulating the discrete problem treated here,
and for her friendly and useful feedback on an earlier version of this manuscript.

\section*{Appendices}

\appendices


\section{Proof of existence of the tilted free energy via subadditivity}\label{app_subadd}

With future use in mind, we consider a more general model of RW in random potential on $\mathbb{Z}^d$ with $d\ge1$. The proof of Theorem \ref{pargoz} that we give below follows \cite[Section 2]{Var2003} closely and does not require any additional effort due to this generality.

Denote by $(X_i)_{i\ge0}$ the SSRW on $\mathbb{Z}^d$ with $Z_{i+1} := X_{i+1} - X_i\in U := \{\pm e_1,\ldots,\pm e_d\}$. Let $(\Omega,\mathcal{F},\mathbb{P})$ be a probability space on which a collection $\{T_z: \Omega\to\Omega\}_{z\in U}$ of invertible measure-preserving transformations act ergodically. Fix a bounded and measurable function $\Psi:\Omega\times U\to\mathbb{R}$. For every $n\in\mathbb{N}\cup\{0\}$, $x,y\in\mathbb{Z}^d$ and $\omega\in\Omega$, define
$$f(n,x,y,\omega) = E_x\left[e^{\sum_{i=0}^{n-1}\Psi(T_{X_i}\omega,Z_{i+1})}\one_{\{X_n = y\}}\right]\qquad\text{and}\qquad F(n,x,\omega) = E_x\left[e^{\sum_{i=0}^{n-1}\Psi(T_{X_i}\omega,Z_{i+1})}\right].$$
Here, $E_x$ stands for expectation with respect to the law of $(X_i)_{i\ge0}$ when $X_0 = x$.

\begin{theorem}\label{pargoz}
	For $\mathbb{P}$-a.e.\ $\omega$, the limit
	\begin{equation}\label{hayitdisi}
	\Lambda(\Psi) = \lim_{n\to\infty}\frac1{n}\log F(n,0,\omega)
	\end{equation}
	exists. Moreover, $\Lambda(\Psi)$ is a deterministic quantity.
\end{theorem}

\begin{proof}
	Assume without loss of generality that $\Psi:\Omega\times U\to[\psi_o,0]$ for some $\psi_o>-\infty$. (Otherwise, we can subtract an appropriate constant from $\Psi$, take the limit in \eqref{hayitdisi}, and add the constant back.)
	For every $c>0$, $t\ge0$, $x,y\in\mathbb{Z}^d$ and $\omega\in\Omega$, define
	\begin{equation}\label{defsifir}
	f_c(t,x,y,\omega) = \sup_{n\ge0}\left[f(n,x,y,\omega)e^{-c|n-t|}\right].
	\end{equation}
	We make several observations. First,
	\begin{equation}\label{ob1}
	-\infty<\log f_c(t,x,y,\omega) \le 0
	\end{equation}
	since $\Psi(\cdot,\cdot) \le 0$. Second, it is clear from \eqref{defsifir} that
	\begin{align}
	&\;\log f_c(t,x,y,\omega) = \log f_c(t,0,y-x,T_x\omega)\qquad\text{and}\label{ob2a}\\
	&|\log f_c(t,x,y,\omega) - \log f_c(t',x,y,\omega)| \le c|t - t'|.\label{ob2b}
	\end{align}
	Third, for every $z,z'\in U$,
	\begin{align*}
	f_c(t,x+z,y+z',\omega) &= \sup_{n\ge0}\left[f(n,x+z,y+z',\omega)e^{-c|n-t|}\right]\\
	&\ge \sup_{n\ge0}\left[f(n+2,x+z,y+z',\omega)e^{-c|n+2-t|}\right]\\
	&\ge \left(\frac{e^{\psi_o}}{2d}\right)^2\sup_{n\ge0}\left[f(n,x,y,\omega)e^{-c|n+2-t|}\right]\\
	&\ge \left(\frac{e^{\psi_o-c}}{2d}\right)^2f_c(t,x,y,\omega)
	\end{align*}
	since the probability of moving from $x+z$ to $x$ (resp.\ from $y$ to $y+z'$) in one step is equal to $\frac1{2d}$. 
	Therefore, there exists a constant $c' = c + |\psi_o| + \log(2d) > 0$ such that
	\begin{equation}\label{ob3}
	|\log f_c(t,x,y,\omega) - \log f_c(t,x',y',\omega)| \le c'\left(|x - x'|_1 + |y - y'|_1\right),
	\end{equation}
	where $|\cdot|_1$ denotes the $\ell_1$-norm on $\mathbb{R}^d$. Fourth, 
	\begin{equation}\label{ob4}
	\log f_c(t+s,0,y,\omega) \ge \log f_c(t,0,x,\omega) + \log f_c(s,x,y,\omega)
	\end{equation}
	since
	\begin{align*}
	f_c(t+s,0,y,\omega) &= \sup_{n\ge0}\left[f(n,0,y,\omega)e^{-c|n-(t+s)|}\right]\\
	&= \sup_{n,m\ge0}\left[f(n+m,0,y,\omega)e^{-c|(n-t)+(m-s)|}\right]\\
	&= \sup_{n,m\ge0}\left[\sum_{x'}f(n,0,x',\omega)f(m,x',y,\omega)e^{-c|(n-t)+(m-s)|}\right]\\
	&\ge \sup_{n,m\ge0}\left[f(n,0,x,\omega)f(m,x,y,\omega)e^{-c|n-t| - c|m-s|}\right]\\
	&= f_c(t,0,x,\omega)f_c(s,x,y,\omega).
	\end{align*}
	It follows from \cite[Theorem 2.1]{Var2003} (which is in turn based on Liggett's subadditive ergodic theorem \cite{Lig1985}) that \eqref{ob1} - \eqref{ob4} ensure the existence of a deterministic, Lipschitz continuous and concave function $\lambda_c:\mathbb{R}^d\to(-\infty,0]$ such that
	\begin{equation}\label{importel}
	\mathbb{P}\left(\lim_{n\to\infty}\frac1{n}\log f_c(n,0,x_n,\omega) = \lambda_c(\xi)\ \text{for every $\xi\in\mathbb{R}^d$ and $(x_n)_{n\ge1}$ such that $\frac{x_n}{n}\to\xi$}\right) = 1.
	\end{equation}
	
	We are ready to establish upper and lower bounds that will imply the existence of the limit in \eqref{hayitdisi}. For every $c>0$ and $\mathbb{P}$-a.e.\ $\omega$,
	\begin{align*}
	\limsup_{n\to\infty}\frac1{n}\log F(n,0,\omega) &= \limsup_{n\to\infty}\frac1{n}\log\sum_{y\in\mathbb{Z}^d:\,|y|_1 \le n} f(n,0,y,\omega) = \limsup_{n\to\infty}\sup_{y\in\mathbb{Z}^d:\,|y|_1 \le n} \frac1{n}\log f(n,0,y,\omega)\\
	&\le \limsup_{n\to\infty}\sup_{y\in\mathbb{Z}^d:\,|y|_1 \le n} \frac1{n}\log f_c(n,0,y,\omega) = \sup_{\xi\in\mathbb{R}^d:\,|\xi|_1 \le 1}\lambda_c(\xi). 
	\end{align*}
	Here, the last equality follows from \eqref{importel} and the continuity of $\xi\mapsto\lambda_c(\xi)$ as in the proof of Varadhan's integral lemma (see \cite[Theorem 4.3.1]{DemZei2010}). It is clear from \eqref{defsifir} that $f_c(\cdot,\cdot,\cdot,\cdot)$ decreases as $c$ increases, and so does $\lambda_c(\xi)$. Therefore, $$\lambda(\xi) := \lim_{c\to\infty}\lambda_c(\xi) = \inf_{c>0}\lambda_c(\xi)\in[-\infty,0]$$ exists. Moreover, $c\mapsto\lambda_c(\xi)$ is convex since it is the limit of the supremum of a collection of linear functions. Using Sion's minimax theorem (see \cite{Kom1988}), we deduce the following upper bound:
	\begin{align*}
	\limsup_{n\to\infty}\frac1{n}\log F(n,0,\omega) &\le \inf_{c>0}\sup_{\xi\in\mathbb{R}^d:\,|\xi|_1 \le 1}\lambda_c(\xi) = \sup_{\xi\in\mathbb{R}^d:\,|\xi|_1 \le 1}\inf_{c>0}\lambda_c(\xi)\\
	&= \sup_{\xi\in\mathbb{R}^d:\,|\xi|_1 \le 1}\lambda(\xi) =: \Lambda(\Psi) \le 0.
	\end{align*}
	
	Obtaining a matching lower bound is equivalent to showing that
	\begin{equation}\label{dizdurs}
	\liminf_{n\to\infty}\frac1{n}\log F(n,0,\omega) \ge \lambda(\xi)
	\end{equation}
	for every $\xi\in\mathbb{R}^d$ such that $|\xi|_1\le 1$. There is nothing to prove if $\lambda(\xi) = - \infty$. (In fact, this case can be ruled out.) Assume $\lambda(\xi) = -\ell > - \infty$. Fix an arbitrary $\epsilon>0$ and choose $c \ge (\ell+1)\epsilon^{-1}$. Then, $\lambda_c(\xi) \ge -\ell$ by monotonicity in $c$. Recalling \eqref{defsifir} and \eqref{importel}, we see that
	$$\sup_{m\ge0}\left[f(m,0,x_n,\omega)e^{-c|m-n|}\right] \ge e^{-n\ell + o(n)}$$
	for $\mathbb{P}$-a.e.\ $\omega$ and every $(x_n)_{n\ge1}$ such that $\frac{x_n}{n}\to\xi$. If $|m-n|\ge n\epsilon$, then $$f(m,0,x_n,\omega)e^{-c|m-n|} \le e^{-nc\epsilon} \le e^{-n(\ell+1)}.$$ Therefore,
	$$e^{-n\ell + o(n)} \le \sup_{m\ge0:\,|m-n|<n\epsilon}\left[f(m,0,x_n,\omega)e^{-c|m-n|}\right] \le \sup_{m\ge0:\,|m-n|<n\epsilon}F(m,0,\omega).$$
	Observe that $|\log F(m,0,\omega) - \log F(n,0,\omega)| \le |m-n||\psi_o|$, which gives
	$$F(n,0,\omega) \ge e^{-n\ell -n\epsilon |\psi_o| + o(n)}.$$
	Since $\epsilon>0$ is arbitrary, the desired lower bound \eqref{dizdurs} follows.
\end{proof}

\begin{remark}\label{furrefvar}
	There are alternative proofs of Theorem \ref{pargoz}. One of the authors established in \cite{Yil2009} a so-called level-2 LDP from the point of view of the particle performing nearest-neighbor random walk in random environment (RWRE) on $\mathbb{Z}^d$, from which the existence of the limit in \eqref{hayitdisi} follows as a corollary by Varadhan's integral lemma. That 
	paper built upon the Ph.D.\ thesis of Rosenbluth \cite{Ros2006} who in turn adapted the work of Kosygina, Rezakhanlou and Varadhan \cite{KosRezVar2006} on the homogenization of second-order HJ stochastic PDEs 
	with convex Hamiltonians.
	This approach is certainly more technical than the short and subadditivity-based proof we gave above, but it has the advantage of providing two variational formulas for $\Lambda(\Psi)$ (see Appendix \ref{app_varfor} for these formulas in our setting). This result was subsequently generalized in \cite{RasSepYil2013} to random walks with arbitrary set of allowed steps (including directed walks). In the latter setting, Rassoul-Agha and Sepp\"al\"ainen \cite{RasSep2014} also gave a proof of existence via subadditivity. Finally, assuming the existence of $\Lambda(\Psi)$, several variational formulas for it were given in \cite{RasSepYil2017} via a potential-theoretic approach which results in much shorter proofs than those in \cite{Ros2006, Yil2009, RasSepYil2013}.
\end{remark}

\section{Centered cocycles and sublinearity of path sums}\label{app_cocycle}

\begin{definition}\label{cencocdef}
A function $F:\Omega\times\{-1,1\}\to\mathbb{R}$ is said to be a cocycle if $F(\cdot,1)$ is $\mathcal{F}$-measurable and $F(\omega,-1) = -F(T_{-1}\omega,1)$ for every $\omega\in\Omega$. $F$ is said to be a centered cocycle if $\mathbb{E}[F(\cdot,1)] = 0$.
\end{definition}
\corO{The set of centered cocycles is denoted by $\mathcal{K}_0$.}

\begin{lemma}\label{cencoclem}
  If \corO{$F\in \mathcal{K}_0$} is  bounded, then for every $B>0$ and $\mathbb{P}$-a.e.\ $\omega$,
$$\lim_{n\to\infty}\frac1{n}\sup\left\{\left|\sum_{i=0}^{n-1}F(T_{x_i}\omega,z_{i+1})\right|:\, |x|\le B,\ x_0 = [nx],\ z_{i+1} = x_{i+1} - x_i = \pm1\right\} = 0.$$
\end{lemma}

\begin{proof}
For every $y\in\mathbb{Z}$, define
$$f(\omega,y) = \begin{cases}\sum_{i=0}^{y-1}F(T_i\omega,1)&\ \text{if $y>0$,}\\0&\ \text{if $y=0$,}\\\sum_{i=y}^{-1}F(T_i\omega,1)&\ \text{if $y<0$.}\end{cases}$$
Since $F$ is bounded, there exists a $K>0$ such that
\begin{equation}\label{lembor1}
|f(\omega,y) - f(\omega,y')| \le K|y-y'|
\end{equation}
for every $y,y'\in\mathbb{Z}$. Since $F$ is centered, $f(\omega,y) = o(|y|)$ for $\mathbb{P}$-a.e.\ $\omega$ by the Birkhoff ergodic theorem. Hence, for every $\ep,B>0$ and $k\in\mathbb{N}$, there exists an $n_0(\omega,\ep,B,k)$ such that
\begin{equation}\label{lembor2}
\frac1{n}\left|f\left(\omega,\frac{jn(B+1)}{k}\right)\right| < \ep
\end{equation}
for every integer $j\in[-k,k]$ and $n\ge n_0$. Combining \eqref{lembor1} and \eqref{lembor2}, we deduce that
\begin{equation}\label{lembor3}
\lim_{n\to\infty}\frac1{n}\sup\{|f(\omega,y)|:\,|y|\le n(B+1)\} = 0.
\end{equation}
Since $F$ is a cocycle, telescoping gives $\sum_{i=0}^{n-1}F(T_{x_i}\omega,z_{i+1}) = f(\omega,x_n) - f(\omega,[nx])$ for any nearest-neighbor path $x_{0,n}$ with $x_0 = [nx]$. Note that $|x_n| \le |[nx]| + n \le n(B+1)$. Therefore, the desired result follows from \eqref{lembor3}. 
\end{proof}

\section{Variational formulas for the tilted free energy}\label{app_varfor}
\corO{We present here two variational formulas for the tilted free energy 
(defined in \eqref{nolimit}) in our one-dimensional nearest-neighbor setting.
These are provided for completeness and are not used elsewhere in the paper,
except that some notation is used also in Appendix \ref{app_isimyok}.}

The variational formulas are
\begin{align}
\Lambda_\beta(\theta) &= \inf_{F\in\mathcal{K}_0}\Pesssup_{\omega}\left\{\beta V(\omega) + \log\left(\frac1{2}e^{\theta + F(\omega,1)} + \frac1{2}e^{- \theta + F(\omega,-1)}\right)\right\}\quad\text{and}\label{Kvarfor}\\
\Lambda_\beta(\theta) &= \sup_{(q,\phi)}\int \left[\beta V(\omega) - I(q(\omega)\,|\,p(\theta))\right]\phi(\omega)d\mathbb{P}(\omega) + \log\cosh(\theta).\label{varyas}
\end{align}
In \eqref{varyas}, $p(\theta) = e^\theta/(e^\theta + e^{-\theta})$,
$$I(q\,|\,p) = q\log(q/p) + (1-q)\log({(1-q)}/{(1-p)}),$$
and the supremum is taken over all $\mathcal{F}$-measurable $q:\Omega\to(0,1)$ and $\phi:\Omega\to(0,\infty)$ such that $\mathbb{E}[\phi(\cdot)] = 1$ and
$$q(T_{-1}\omega)\phi(T_{-1}\omega) + (1 - q(T_1\omega))\phi(T_1\omega) = \phi(\omega)$$
for $\mathbb{P}$-a.e.\ $\omega$. The last equality implies that the probability measure $\phi d\mathbb{P}$ is invariant for the so-called environment Markov chain $(T_{X_i}\omega)_{i\ge0}$ induced by the RWRE with probability $q(T_x\omega)$ of jumping to the right at the point $x$ in the environment $\omega$. 
\corO{These variational formulas follow e.g.\ from \cite[Theorem 2.1]{Yil2009}. See
Remark \ref{furrefvar} for further references.}

When $\Lambda_\beta(\theta) > \beta$, the variational problems in \eqref{Kvarfor} and \eqref{varyas} are solved as follows. Assume without loss of generality that $\theta>0$. (Recall from Proposition \ref{temelsaf}(c) that $\Lambda_\beta(0) = \beta$.) Define $$q_{\beta,\theta}(\omega) = \frac1{2}e^{\beta V(\omega) + \theta + F_{\beta,\theta}(\omega,1) - \Lambda_\beta(\theta)}.$$
Then, \eqref{menzilde} readily implies
$$1 - q_{\beta,\theta}(\omega) = \frac1{2}e^{\beta V(\omega) - \theta + F_{\beta,\theta}(\omega,-1) - \Lambda_\beta(\theta)}\quad\text{and}\quad r_{\beta,\theta}(\omega) := \frac{1-q_{\beta,\theta}(\omega)}{q_{\beta,\theta}(\omega)} = e^{-2\theta - F_{\beta,\theta}(T_{-1}\omega,1) - F_{\beta,\theta}(\omega,1)}.$$
Note that
\begin{equation}\label{nonnest}
0< \left(E_0\left[e^{-\Lambda_\beta(\theta)\tau_1}\one_{\{\tau_1<\infty\}}\right]\right)^2 \le r_{\beta,\theta}(\omega) < e^{2(\beta - \Lambda_\beta(\theta))} < 1
\end{equation}
by \eqref{fboundsrec}. 
We consider the RWRE with probability $q_{\beta,\theta}(T_x\omega)$ of jumping to the right at the point $x$ in the environment $\omega$. It induces a probability measure $\hat P_0^{\beta,\theta,\omega}$ on nearest-neighbor paths starting at $0$. $\hat E_0^{\beta,\theta,\omega}$ denotes expectation under $\hat P_0^{\beta,\theta,\omega}$. With this notation, it is clear from \eqref{nonnest} that $\mathbb{E}\left[\hat E_0^{\beta,\theta,\omega}[\tau_1]\right] < \infty$. Therefore,
$$\psi_{\beta,\theta}(\omega) := \sum_{i=0}^\infty \hat P_0^{\beta,\theta,\omega}(X_i = 0)$$
satisfies $\mathbb{E}[\psi_{\beta,\theta}(\cdot)] < \infty$ (see \cite[Theorem 5.17]{Yil2009}) and we define
$$\phi_{\beta,\theta}(\omega) = \frac{\psi_{\beta,\theta}(\omega)}{\mathbb{E}[\psi_{\beta,\theta}(\cdot)]}.$$

\begin{proposition}
	Assume \eqref{ass_wlog} and \eqref{ass_pan}. If $\theta>0$ and $\Lambda_\beta(\theta) > \beta$, then
	\begin{itemize}
		\item [(a)] the infimum in \eqref{Kvarfor} is attained at $F_{\beta,\theta}$, and
		\item [(b)] the supremum in \eqref{varyas} is attained at $(q_{\beta,\theta},\phi_{\beta,\theta})$.
	\end{itemize}
\end{proposition}

\begin{proof}
	(a) This follows immediately from \eqref{menzilde}.
	
	(b) It is easy to show that $$q_{\beta,\theta}(T_{-1}\omega)\phi_{\beta,\theta}(T_{-1}\omega) + (1 - q_{\beta,\theta}(T_1\omega))\phi_{\beta,\theta}(T_1\omega) = \phi_{\beta,\theta}(\omega)$$
	for $\mathbb{P}$-a.e.\ $\omega$ (see \cite[Theorem 5.17]{Yil2009}). By Kozlov's lemma (see \cite{Koz1985}), the probability measure $\phi_{\beta,\theta}d\mathbb{P}$ is ergodic for the RWRE defined by $q_{\beta,\theta}$. Therefore, for $\mathbb{P}$-a.e.\ $\omega$ and $\hat P_0^{\beta,\theta,\omega}$-a.s.,
	\begin{equation}\label{gelmy}
	\int \left[q_{\beta,\theta}(\omega)F_{\beta,\theta}(\omega,1) + (1 - q_{\beta,\theta}(\omega))F_{\beta,\theta}(\omega,-1)\right]\phi_{\beta,\theta}(\omega)d\mathbb{P}(\omega) = \lim_{n\to\infty}\frac1{n}\sum_{i=0}^{n-1}F_{\beta,\theta}(T_{X_i}\omega,Z_{i+1}) = 0.
	\end{equation}
	Here, the first equality holds by the Birkhoff ergodic theorem, and the second equality follows from Lemma \ref{cencoclem}. We use \eqref{gelmy} to deduce that
	\begin{align*}
	&\int \left[\beta V(\omega) - I(q_{\beta,\theta}(\omega)\,|\,p(\theta))\right]\phi_{\beta,\theta}(\omega)d\mathbb{P}(\omega) + \log\cosh(\theta)\\
	&\quad = \int \left[\beta V(\omega) - q_{\beta,\theta}(\omega)\{\beta V(\omega) + F_{\beta,\theta}(\omega,1) - \Lambda_\beta(\theta) + \log\cosh(\theta)\}\right.\\
	&\qquad \qquad - \left.(1 - q_{\beta,\theta}(\omega))\{\beta V(\omega) + F_{\beta,\theta}(\omega,-1) - \Lambda_\beta(\theta) + \log\cosh(\theta)\}\right]\phi_{\beta,\theta}(\omega)d\mathbb{P}(\omega) + \log\cosh(\theta)\\
	&\quad = \Lambda_\beta(\theta) - \int \left[q_{\beta,\theta}(\omega)F_{\beta,\theta}(\omega,1) + (1 - q_{\beta,\theta}(\omega))F_{\beta,\theta}(\omega,-1)\right]\phi_{\beta,\theta}(\omega)d\mathbb{P}(\omega) = \Lambda_\beta(\theta).\qedhere
	\end{align*}
\end{proof}

\begin{remark}\label{correctorname}
	The components of the vector $\left(\frac1{2}e^{\beta V(\omega) + \theta - \Lambda_\beta(\theta)}, \frac1{2}e^{\beta V(\omega) - \theta - \Lambda_\beta(\theta)}\right)$ do not add up to $1$.
	$F_{\beta,\theta}$ is called the corrector precisely because it enables us to modify this vector and obtain the transition kernel $(q_{\beta,\theta}(\omega),1-q_{\beta,\theta}(\omega))$. The sublinearity in Lemma \ref{cencoclem} 
	ensures that this modification does not alter the
	asymptotics in the $\frac1{n}\log$ scale. 
\corO{We refer to  \cite[Section 2, Remark 2]{BerMukOka_preprint} for a 
discussion of the relation between these type of correctors and the classical Kipnis-Varadhan correctors used to construct martingales in the proofs
of invariance principles for Markov processes.}
\end{remark}

\section{Nondifferentiability of the tilted free energy at the endpoints of $\{\theta\in\mathbb{R}:\,\Lambda_\beta(\theta) = \beta\}$}\label{app_isimyok}

Recall the notation from the previous section. When $\theta>0$ and $\lambda = \Lambda_\beta(\theta)>\beta$, the equality in \eqref{ustdest} can be expressed as follows:
\begin{align}
\left.\frac{\partial}{\partial\lambda}\mathbb{E}[F_{\beta,\theta}^\lambda(\cdot,1)]\right|_{\lambda = \Lambda_\beta(\theta)} &= \mathbb{E}\left[E_0\left[\tau_1e^{\beta\sum_{i=0}^{\tau_1-1}V(T_{X_i}\omega) - \Lambda_\beta(\theta)\tau_1}\one_{\{\tau_1<\infty\}}\right]e^{\theta + F_{\beta,\theta}(\omega,1)}\right]\nonumber\\&= \mathbb{E}\left[E_0\left[\tau_1e^{\sum_{i=0}^{\tau_1-1}\left\{\beta V(T_{X_i}\omega) + \theta Z_{i+1} + F_{\beta,\theta}(T_{X_i}\omega,Z_{i+1}) - \Lambda_\beta(\theta)\right\}}\one_{\{\tau_1<\infty\}}\right]\right]\nonumber\\
&=\mathbb{E}\left[\hat E_0^{\beta,\theta,\omega}[\tau_1\one_{\{\tau_1<\infty\}}]\right],\label{sadyeng}
\end{align}
\corO{where the second equality is due to telescoping.}
Define $$S_{\beta,\theta}(\omega) = 1 + \sum_{n=1}^\infty\prod_{j=1}^nr_{\beta,\theta}(T_j\omega).$$
Then,
$$\mathbb{E}[S_{\beta,\theta}(\cdot)] < 1+ \sum_{n=1}^\infty\prod_{j=1}^n e^{2(\beta - \Lambda_\beta(\theta))} = \left(1- e^{2(\beta - \Lambda_\beta(\theta))}\right)^{-1} < \infty$$
by \eqref{nonnest}. Therefore, the law of large numbers holds for this RWRE and the limiting velocity $v_{\beta,\theta}$ satisfies
$$(v_{\beta,\theta})^{-1} = \mathbb{E}\left[\hat E_0^{\beta,\theta,\omega}[\tau_1\one_{\{\tau_1<\infty\}}]\right] = \mathbb{E}[(1 + r_{\beta,\theta}(\cdot))S_{\beta,\theta}(\cdot)]$$
(see \cite[Theorem 4.1]{Ali1999}). Recalling \eqref{sadyeng} and the proof of Proposition \ref{temelsaf}(e), we deduce by the implicit function theorem 
\corO{(together with $d /d\theta\left(\mathbb{E}\left[
F_{\beta,\theta}^{\Lambda_\beta(\theta)}(\cdot,1)\right]\right)=0$ 
due to \eqref{eq-star14} and 
$\partial/\partial \theta (F_{\beta,\theta}^\lambda(\omega,1))=-1$ due to  
\eqref{pregaztap})} that
\begin{equation}\label{derengel}
\frac{d}{d\theta}\Lambda_\beta(\theta) = \left(\left.\frac{\partial}{\partial\lambda}\mathbb{E}[F_{\beta,\theta}^\lambda(\cdot,1)]\right|_{\lambda = \Lambda_\beta(\theta)}\right)^{-1} = v_{\beta,\theta} = \left(\mathbb{E}[(1 + r_{\beta,\theta}(\cdot))S_{\beta,\theta}(\cdot)]\right)^{-1}.
\end{equation}

In order to prove that $\theta\mapsto\Lambda_\beta(\theta)$ is not differentiable at $\pm\theta_b$ defined by $$\theta_b := \sup\{\theta\in\mathbb{R}:\,\Lambda_\beta(\theta)=\beta\} = \inf\{\theta\in\mathbb{R}:\,\theta>0\ \text{and}\ \Lambda_\beta(\theta)>\beta\},$$
it suffices to obtain an upper bound for $\mathbb{E}[(1 + r_{\beta,\theta}(\cdot))S_{\beta,\theta}(\cdot)]$ that is uniform in $\theta>\theta_b$. To this end, observe that 
\begin{align}
  \label{eq-star28}
r_{\beta,\theta}(\omega) &= e^{-2\theta - F_{\beta,\theta}(T_{-1}\omega,1) - F_{\beta,\theta}(\omega,1)} \le e^{-\theta - F_{\beta,\theta}(\omega,1)} = E_0\left[e^{\beta\sum_{i=0}^{\tau_1-1}V(T_{X_i}\omega) - \Lambda_\beta(\theta)\tau_1}\one_{\{\tau_1<\infty\}}\right]\nonumber\\
&= E_0\left[e^{\beta\sum_{i=0}^{\tau_1-1}V(T_{X_i}\omega) - \Lambda_\beta(\theta)\tau_1}\one_{\{\tau_1=1\}}\right] + E_0\left[e^{\beta\sum_{i=0}^{\tau_1-1}V(T_{X_i}\omega) - \Lambda_\beta(\theta)\tau_1}\one_{\{2\le\tau_1<\infty\}}\right]\nonumber\\
&\le \frac1{2}e^{\beta V(\omega) - \Lambda_\beta(\theta)}  + \frac1{2} \le \frac1{2}\left[e^{\beta (V(\omega) - 1)} + 1\right] =: \bar r_\beta(\omega).
\end{align}
Since $V$ takes values in $[0,1]$, we have $\bar r_\beta(\omega)\le 1$ for every $\omega\in\Omega$. Moreover, \eqref{ass_wlog} implies that $\mathbb{E}[\bar r_\beta(\cdot)] < 1$. As the proof of the following warm-up result demonstrates, the advantage of working with $\bar r_\beta(\omega)$ instead of $r_{\beta,\theta}(\omega)$ is that the former does not depend on $\theta$ and it depends on the potential only through $V(\omega)$.

\begin{proposition}\label{nondifiid}
	If $\left(V(T_j\omega)\right)_{j\in\mathbb{Z}}$ are i.i.d.\ under $\mathbb{P}$, then
	$$\mathbb{E}[(1 + r_{\beta,\theta}(\cdot))S_{\beta,\theta}(\cdot)] \le \frac{1 + \mathbb{E}[\bar r_\beta(\cdot)]}{1 - \mathbb{E}[\bar r_\beta(\cdot)]} < \infty$$
	for every $\theta>\theta_b$. Hence, $\theta\mapsto\Lambda_\beta(\theta)$ is not differentiable at $\pm\theta_b$.
\end{proposition}

\begin{proof}
	Since $\bar r_\beta(\omega)$ is a function of $V(\omega)$, the random variables $\left(\bar r_\beta(T_j\omega)\right)_{j\in\mathbb{Z}}$ are i.i.d.\ under $\mathbb{P}$, too. Therefore, \corO{using that
	  $r_{\beta,\theta}(\cdot)\leq \bar r_\beta(\cdot)$ by \eqref{eq-star28},}
	\begin{align*}
	\mathbb{E}[(1 + r_{\beta,\theta}(\cdot))S_{\beta,\theta}(\cdot)] &\le \mathbb{E}\left[(1 + \bar r_\beta(\cdot))\left(1 + \sum_{n=1}^\infty\prod_{j=1}^n\bar r_\beta(T_j\cdot)\right)\right]\\
	&= \mathbb{E}[1 + \bar r_\beta(\cdot)]\sum_{n=0}^\infty\left(\mathbb{E}[\bar r_\beta(\cdot)]\right)^n = \frac{1 + \mathbb{E}[\bar r_\beta(\cdot)]}{1 - \mathbb{E}[\bar r_\beta(\cdot)]} < \infty
	\end{align*}
	whenever $\theta>\theta_b$. We use \eqref{derengel} to deduce that
	\begin{equation}\label{pazardagel}
	\inf\left\{\frac{d}{d\theta}\Lambda_\beta(\theta):\ \theta>\theta_b\right\} \ge \frac{1 - \mathbb{E}[\bar r_\beta(\cdot)]}{1 + \mathbb{E}[\bar r_\beta(\cdot)]} > 0.
	\end{equation}
	
	Recall Proposition \ref{temelsaf}. If the even and convex map $\theta\mapsto\Lambda_\beta(\theta)$ were differentiable at $\pm\theta_b$, it would be continuously differentiable (by the monotonocity of the derivative and an application of Darboux's theorem). However, the latter is ruled out by \eqref{pazardagel} and the fact that the derivative vanishes on the nonempty interior of the closed interval $\{\theta\in\mathbb{R}:\,\Lambda_\beta(\theta)=\beta\}$. This concludes the proof.
\end{proof}

We can relax the i.i.d.\ assumption in Proposition \ref{nondifiid}. To this end, let $$A_h = \{\omega\in\Omega: V(\omega)\le h\} = \{\omega\in\Omega: \bar r_\beta(\omega)\le a(\beta,h)\}$$ for every $h\in(0,1)$, where $$a(\beta,h) := \frac1{2}\left[e^{\beta (h - 1)} + 1\right] < 1.$$
We can use this event to introduce a stationary discrete point process $N = N(\omega) = (N_n(\omega))_{n\ge0}$ with $N_0(\omega) = 0$ and $N_n(\omega) = \sum_{j=1}^n\one_{A_h}(T_j\omega)$.
For every $k\ge0$, define $$R_{h,k} = \min\{n\ge0:\,N_n = k\}\quad\text{and}\quad t_{h,k+1} = R_{h,k+1} - R_{h,k}.$$
In particular,
$$t_{h,1} = \min\{j\ge1:\,V(T_j\omega) \le h\}.$$

\begin{lemma}\label{yenlek}
	$\mathbb{E}[t_{h,1}] = \frac1{2}\mathbb{P}(A_h)\mathbb{E}\left[\left.t_{h,1}(t_{h,1} + 1)\,\right|\,A_h\right]$ and $\mathbb{E}[t_{h,k}] \le \mathbb{P}(A_h)\mathbb{E}\left[\left.t_{h,1}^2\,\right|\,A_h\right]$ for every $k\ge2$.
\end{lemma}

\begin{proof}
	The sequence $(t_{h,k})_{k\ge1}$ is stationary under $\mathbb{P}(\cdot\,|\,A_h)$. Moreover,
	$$\mathbb{E}[g(N(\cdot))] = \mathbb{P}(A_h)\mathbb{E}\left[\left.\sum_{j=0}^{t_{h,1}-1}g(N(T_j\cdot))\,\right|\,A_h\right]$$
	for any nonnegative measurable function $g$ of the point process $N = (N_n)_{n\ge0}$ (see \cite[Theorem 13.3.I]{DalVer2008}). Therefore,
	\begin{align*}
	\mathbb{E}[t_{h,1}] &= \mathbb{P}(A_h)\mathbb{E}\left[\left.\sum_{j=0}^{t_{h,1}-1}(t_{h,1} - j)\,\right|\,A_h\right] = \frac1{2}\mathbb{P}(A_h)\mathbb{E}\left[\left.t_{h,1}(t_{h,1} + 1)\,\right|\,A_h\right]\quad\text{and}\\ 
	\mathbb{E}[t_{h,k}] &= \mathbb{P}(A_h)\mathbb{E}\left[\left.\sum_{j=0}^{t_{h,1}-1}t_{h,k}\,\right|\,A_h\right] = \mathbb{P}(A_h)\mathbb{E}\left[\left.t_{h,1}t_{h,k}\,\right|\,A_h\right]\\
	&\le \mathbb{P}(A_h)\sqrt{\mathbb{E}\left[\left.t_{h,1}^2\,\right|\,A_h\right]\mathbb{E}\left[\left.t_{h,k}^2\,\right|\,A_h\right]} = \mathbb{P}(A_h)\mathbb{E}\left[\left.t_{h,1}^2\,\right|\,A_h\right]
	\end{align*}
	for every $k\ge2$.
\end{proof}

\begin{theorem}\label{nondifgen}
	If $\mathbb{E}[t_{h,1}]<\infty$ for some $h\in(0,1)$, then $\theta\mapsto\Lambda_\beta(\theta)$ is not differentiable at $\pm\theta_b$.
\end{theorem}

\begin{proof}
	If $\mathbb{E}[t_{h,1}]<\infty$ for some $h\in(0,1)$, then $$\mathbb{E}[t_{h,k}] \le \mathbb{P}(A_h)\mathbb{E}\left[\left.t_{h,1}^2\,\right|\,A_h\right] \le 2\mathbb{E}[t_{h,1}]<\infty$$ for every $k\ge2$ by Lemma \ref{yenlek}. Therefore,
	\begin{align*}
	\frac1{2}\mathbb{E}[(1 + r_{\beta,\theta}(\cdot))S_{\beta,\theta}(\cdot)] &\le \mathbb{E}\left[\left(1 + \sum_{n=1}^\infty\prod_{j=1}^n\bar r_\beta(T_j\cdot)\right)\right] = 1 + \sum_{n=1}^\infty\mathbb{E}\left[\prod_{j=1}^n\bar r_\beta(T_j\cdot)\right]\\
	&\le \sum_{n=0}^\infty\mathbb{E}\left[a(\beta,h)^{N_n}\right] = \sum_{n=0}^\infty\sum_{k=0}^n a(\beta,h)^k\mathbb{P}(N_n = k) = \sum_{k=0}^\infty a(\beta,h)^k \sum_{n=k}^\infty \mathbb{P}(N_n = k)\\
	&= \sum_{k=0}^\infty a(\beta,h)^k \sum_{n=k}^\infty \mathbb{E}\left[\one_{\{R_{h,k}\le n < R_{h,k+1}\}}\right] = \sum_{k=0}^\infty a(\beta,h)^k \mathbb{E}\left[\sum_{n=k}^\infty\one_{\{R_{h,k}\le n < R_{h,k+1}\}}\right]\\
	&= \sum_{k=0}^\infty a(\beta,h)^k \mathbb{E}\left[t_{h,k+1}\right] \le \frac{2\mathbb{E}[t_{h,1}]}{1-a(\beta,h)} < \infty
	\end{align*}
	whenever $\theta>\theta_b$, \corO{where \eqref{eq-star28} was
      used in the first inequality}. The rest of the proof is identical to that of Proposition \ref{nondifiid}.
\end{proof}

\section{Large deviation estimates for the number of left excursions of RWs}\label{app_simple}

Let $(X_i)_{i\ge0}$ denote SSRW on $\mathbb{Z}$. Similar to $\tau_y = \inf\{i\ge 0:\,X_i = y\}$ with $y\in\mathbb{Z}$, define
$$\tau_{y-1,y} = \inf\{i\ge1:\,X_{i-1} = y-1, X_i = y\}.$$ 

\begin{lemma}
	For every $\lambda \ge 0$,
	\begin{align}
	&E_0[e^{-\lambda\tau_1}\one_{\{\tau_1<\infty\}}] = e^\lambda - \sqrt{e^{2\lambda} - 1}\qquad\text{and}\label{duzel1}\\
	&E_0[e^{-\lambda\tau_{-1,0}}\one_{\{\tau_{-1,0}<\infty\}}] = \frac{1 - \sqrt{1 - e^{-2\lambda}}}{1 + \sqrt{1 - e^{-2\lambda}}}.\label{duzel2}
	\end{align}
\end{lemma}

\begin{proof}
	The desired equalities clearly hold when $\lambda=0$. For $\lambda>0$ and $x\in\mathbb{Z}$, let $v_\lambda(x) = E_x[e^{-\lambda\tau_0}\one_{\{\tau_0<\infty\}}]$. Then,
	\begin{equation}\label{rakkasmey}
	v_\lambda(x) = \frac1{2}e^{-\lambda}(v_\lambda(x-1) + v_\lambda(x+1))\quad\text{for $x\ne 0$,}
	\end{equation}
	$v_\lambda(0) = 1$, and $\lim_{|x|\to\infty}v_\lambda(x) = 0$. We substitute $v_\lambda(x) = e^{-a|x|}$ into \eqref{rakkasmey} and find after an elementary computation that $e^{\pm a} = e^\lambda \pm \sqrt{e^{2\lambda} - 1}$.
	Consequently, \eqref{duzel1} follows and
	$$E_0[e^{-\lambda\tau_{-1,0}}\one_{\{\tau_{-1,0}<\infty\}}] = E_0[e^{-\lambda\tau_{-1}}\one_{\{\tau_{-1}<\infty\}}]\,E_{-1}[e^{-\lambda\tau_{0}}\one_{\{\tau_{0}<\infty\}}] = e^{-2a} = \frac{1 - \sqrt{1 - e^{-2\lambda}}}{1 + \sqrt{1 - e^{-2\lambda}}}.\qedhere$$
\end{proof}

\begin{proof}[Proof of Proposition \ref{countexcursion}]
	Recall from \eqref{elemtere} that $\mathcal{L}_0(X_{0,n})$ counts the number of complete left excursions of SSRW starting from the origin,
	\corOO{up to time $n$}. For every $\xi\ge0$, define
	$$I(\xi) = - \lim_{n\to\infty}\frac1{n}\log P_0(\mathcal{L}_0(X_{0,n}) \ge n \xi).$$
	It is clear that $I(0) = 0$ and $I(\xi) = \infty$ if $\xi>1/2$. For $\xi\in(0,1/2]$,
	\begin{equation}
	\begin{aligned}\label{melmer}
	I(\xi) 
	&\;= - \lim_{n\to\infty}\frac1{n}\log P_0\left(\sum_{i=1}^{n\xi} \tau_{-1,0}^i \le n\right) = - \xi\lim_{n\to\infty}\frac1{n\xi}\log P_0\left(\frac1{n\xi}\sum_{i=1}^{n\xi} \tau_{-1,0}^i \le \frac1{\xi}\right)\\
	&\;= \xi\inf_{0\le a\le \frac1{\xi}}\sup_{\lambda\in\mathbb{R}}\left\{-\lambda a  - \log E_0[e^{-\lambda\tau_{-1,0}}\one_{\{\tau_{-1,0}<\infty\}}]\right\}\\
	&\;= \xi\inf_{0\le a\le \frac1{\xi}}\sup_{\lambda\ge 0}\left\{-\lambda a  - \log E_0[e^{-\lambda\tau_{-1,0}}\one_{\{\tau_{-1,0}<\infty\}}]\right\}\\
	&\;= -\inf_{\lambda \ge 0}\left\{\lambda + \xi \log E_0[e^{-\lambda\tau_{-1,0}}\one_{\{\tau_{-1,0}<\infty\}}]\right\}
	\end{aligned}
	\end{equation}
	by Cram\'er's theorem (see \cite[Theorem 2.2.3]{DemZei2010}), where $\tau_{-1,0}^i$ are independent copies of $\tau_{-1,0}$. We substitute the expression on the right-hand side of \eqref{duzel2} into the last expression in \eqref{melmer}, check that the infimum there is attained when $\sqrt{1 - e^{-2\lambda}} = 2\xi$, and obtain the following formula (with the convention $0\log0 = 0$):
	$$I(\xi) = \left(\frac{1-2\xi}{2}\right)\log(1 - 2\xi) + \left(\frac{1+2\xi}{2}\right)\log(1 + 2\xi) > 0.$$
	Note that $I$ is continuous and strictly increasing on $[0,1/2]$. Set $I(\xi) = \infty$ for every $\xi<0$. It follows that $\left(P_0(\frac1{n}\mathcal{L}_0(X_{0,n}) \in \cdot\,)\right)_{n\ge1}$ satisfies the large deviation principle with rate function $I:\mathbb{R}\to[0,\infty]$. 
	Finally,
	$$J(2c) = \lim_{n\to\infty}\frac1{n}\log E_0\left[e^{2c\mathcal{L}_0(X_{0,n})}\right] = \sup_{0\le\xi\le1/2}\{2c\xi - I(\xi)\} = \log\cosh(c)$$
	by Varadhan's integral lemma (see \cite[Theorem 4.3.1]{DemZei2010}) and a routine computation.
	
	Recall from Section \ref{cinsisimhay} that $(Y_k^\ell)_{k\ge0}$ is a reflected RW on $[-\ell,\ell-1]$ subject to geometric holding times (with rate $1/2$) at $-\ell$ and $\ell-1$. Its transition probabilities are given in \eqref{tantana}. For $y\in\mathbb{Z}$, let
	$$\tilde\tau_y^\ell = \inf\{k\ge 0:\,Y_k^\ell = y\}\quad\text{and}\quad\tilde\tau_{y-1,y}^\ell = \inf\{k\ge1:\,Y_{k-1}^\ell = y-1, Y_k^\ell = y\}.$$
	For $\lambda > 0$, let $\tilde v_\lambda^\ell(x) = E_x[e^{-\lambda\tilde\tau_0^\ell}\one_{\{\tilde\tau_{0}^\ell<\infty\}}]$. Then, similar to \eqref{rakkasmey},
	$$\tilde v_\lambda^\ell(x) = \frac1{2}e^{-\lambda}(\tilde v_\lambda^\ell(x-1) + \tilde v_\lambda^\ell(x+1))\quad\text{for $x\in[-\ell-1,-1]\cup[1,\ell-2]$,}$$
	$\tilde v_\lambda^\ell(0) = 1$, and $\tilde v_\lambda^\ell(x) \le e^{-\lambda|x|}$ for $x\in[-\ell,\ell-1]$.
	By the maximum principle, $\tilde v_\lambda^\ell(x) - v_\lambda(x)\to0$ as $\ell\to\infty$.
	Therefore,
	\begin{align*}
	\lim_{\ell\to\infty}E_0[e^{-\lambda\tilde\tau_{-1,0}^\ell}\one_{\{\tilde\tau_{-1,0}^\ell<\infty\}}] &= \lim_{\ell\to\infty}\left(E_0[e^{-\lambda\tilde\tau_{-1}^\ell}\one_{\{\tilde\tau_{-1}^\ell<\infty\}}]\,E_{-1}[e^{-\lambda\tilde\tau_{0}^\ell}\one_{\{\tilde\tau_{0}^\ell<\infty\}}]\right)\\
	&= E_0[e^{-\lambda\tau_{-1}}\one_{\{\tau_{-1}<\infty\}}]\,E_{-1}[e^{-\lambda\tau_{0}}\one_{\{\tau_{0}<\infty\}}]\\
	&= E_0[e^{-\lambda\tau_{-1,0}}\one_{\{\tau_{-1,0}<\infty\}}].
	\end{align*}
	We record this as follows: for every $\lambda>0$,
	$$\varphi_\ell(\lambda) := \log E_0[e^{-\lambda\tilde\tau_{-1,0}^\ell}\one_{\{\tilde\tau_{-1,0}^\ell<\infty\}}]\quad\text{and}\quad\varphi(\lambda) := \log E_0[e^{-\lambda\tau_{-1,0}}\one_{\{\tau_{-1,0}<\infty\}}]$$
	satisfy
	\begin{equation}\label{elher}
	\lim_{\ell\to\infty}\varphi_\ell(\lambda) = \varphi(\lambda).
	\end{equation}

	For every $\ell,m\in\mathbb{N}$ and $\lambda>0$,	
	\begin{align*}
	E_0\left[e^{-\varphi_\ell(\lambda)\mathcal{L}_0(Y_{0,m}^\ell)}\right] &= \sum_{k=0}^{[m/2]}e^{-k\varphi_\ell(\lambda)}P_0\left(\mathcal{L}_0(Y_{0,m}^\ell) = k\right)\\
	&= \sum_{k=0}^{[m/2]}e^{-k\varphi_\ell(\lambda)}P_0\left(\sum_{i=1}^k \tilde\tau_{-1,0}^{\ell,i} \le m < \sum_{i=1}^{k+1} \tilde\tau_{-1,0}^{\ell,i}\right)\\
	&\le \sum_{k=0}^{[m/2]}e^{-k\varphi_\ell(\lambda) + m\lambda}E_0\left[e^{-\lambda\sum_{i=1}^k \tilde\tau_{-1,0}^{\ell,i}}\one_{\left\{\sum_{i=1}^k \tilde\tau_{-1,0}^{\ell,i} \le m < \sum_{i=1}^{k+1} \tilde\tau_{-1,0}^{\ell,i}\right\}}\right]\\
	&\le \sum_{k=0}^{[m/2]}e^{-k\varphi_\ell(\lambda) + m\lambda}E_0\left[e^{-\lambda\sum_{i=1}^k \tilde\tau_{-1,0}^{\ell,i}}\one_{\left\{\sum_{i=1}^k \tilde\tau_{-1,0}^{\ell,i} < \infty\right\}}\right]\\
	&= \sum_{k=0}^{[m/2]}e^{-k\varphi_\ell(\lambda) + m\lambda + k\varphi_\ell(\lambda)} = (1 + [m/2])e^{m\lambda},
	\end{align*}
	where $\tilde\tau_{-1,0}^{\ell,i}$ are independent copies of $\tilde\tau_{-1,0}^\ell$. Therefore,
	$$J_\ell(-\varphi_\ell(\lambda)) = \limsup_{m\to\infty}\frac1{m}\log E_0\left[e^{-\varphi_\ell(\lambda)\mathcal{L}_0(Y_{0,m}^\ell)}\right] \le \lambda.$$
	It follows from \eqref{duzel2} that $\varphi(\log\cosh(c)) = -2c$. Since $J_\ell$ is clearly Lipschitz continuous with Lipschitz constant $1/2$,
	\begin{align*}
	J_\ell(2c) &\le J_\ell(-\varphi_\ell(\log\cosh(c))) + \frac1{2}|\varphi_\ell(\log\cosh(c)) - \varphi(\log\cosh(c))|\\
	&\le \log\cosh(c) + \frac1{2}|\varphi_\ell(\log\cosh(c)) - \varphi(\log\cosh(c))|.
	\end{align*}
	Recalling \eqref{elher}, we deduce that
	\begin{equation}\label{guzolbir}
	\limsup_{\ell\to\infty}J_\ell(2c) \le \log\cosh(c).
	\end{equation}
	On the other hand, it is clear from the definition of $(Y_k^\ell)_{k\ge0} = (X_{\sigma_k})_{k\ge0}$ that $\mathcal{L}_0(Y_{0,m}^\ell) \ge \mathcal{L}_0(X_{0,m})$ for every $m\ge1$ and every realization of the SSRW path $X_{0,m}$. Therefore,
	\begin{equation}\label{guzoliki}
	J_\ell(2c) = \limsup_{m\to\infty}\frac1{m}\log E_0\left[e^{2c\mathcal{L}_0(Y_{0,m}^\ell)}\right] \ge \lim_{m\to\infty}\frac1{m}\log E_0\left[e^{2c\mathcal{L}_0(X_{0,m})}\right] = J(2c) = \log\cosh(c)
	\end{equation}
	for every $\ell\in\mathbb{N}$. Combining \eqref{guzolbir} and \eqref{guzoliki} concludes the proof.
\end{proof}

\bibliographystyle{abbrv}
\bibliography{control_references}

\end{document}